\documentclass[final]{siamltex}

\newcommand{\field}[1]{\mathbb{#1}}
\newcommand{\C}{\field{C}}
\newcommand{\R}{\field{R}}

\newcommand{\be}{\begin{equation}}
\newcommand{\ee}{\end{equation}}
\newcommand{\I}{{\mathrm{i}}}
\usepackage{epsfig}
\usepackage{graphicx}
\usepackage{float}
\usepackage[english]{babel}
\usepackage{amsmath}
\usepackage{amssymb}

\title{Application of the inhomogeneous Lippmann-Schwinger equation to inverse scattering problems}

\author{Giovanni Giorgi\thanks{Dipartimento di Matematica, Universit\`a di Genova, via Dodecaneso 35, I-16146 Genova, Italy (giorgi$@$dima.unige.it).}\and Massimo Brignone\thanks{Azienda Ospedaliera Universitaria San Martino, Largo Benzi 10, I-16132 Genova, Italy (brignone$@$dima.unige.it).}\and Riccardo Aramini\thanks{Dipartimento di Matematica, Universit\`a di Genova, via Dodecaneso 35, I-16146 Genova, Italy (aramini$@$dima.unige.it).}\and Michele Piana\thanks{Dipartimento di Matematica, Universit\`a di Genova, via Dodecaneso 35, I-16146 Genova, Italy and CNR - SPIN, Genova, via Dodecaneso 33, I-16146 Genova, Italy (piana$@$dima.unige.it).
}}

\begin{document}

\maketitle

\begin{abstract}
In this paper we present a hybrid approach to numerically solve two-dimensional electromagnetic inverse scattering problems,
whereby the unknown scatterer is hosted by a possibly inhomogeneous background. The approach is `hybrid' in that it merges a
qualitative and a quantitative method to optimize the way of exploiting the \textit{a priori} information on the
background within the inversion procedure, thus improving the quality of the reconstruction and reducing the data amount necessary
for a satisfactory result. In the qualitative step, this \textit{a priori} knowledge is utilized to implement the linear sampling method in its
near-field formulation for an inhomogeneous background, in order to identify the region where the scatterer is located.
On the other hand, the same \textit{a priori} information is also encoded in the quantitative step by extending and applying the contrast source
inversion method
to what we call the `inhomogeneous Lippmann-Schwinger equation': the latter is a generalization of the classical Lippmann-Schwinger equation to the case of an inhomogeneous background, and in our paper is deduced from the differential formulation of the direct scattering problem to provide the
reconstruction algorithm with an appropriate theoretical basis. Then, the point values of the refractive index are computed only in the region identified by the linear sampling method at the previous step. The effectiveness of this hybrid approach is supported by numerical
simulations presented at the end of the paper.
\end{abstract}

\begin{keywords}
Inverse scattering, Lippmann-Schwinger equation, hybrid methods.
\end{keywords}

\begin{AMS}
78A46, 45Q05
\end{AMS}

\pagestyle{myheadings}
\thispagestyle{plain}
\markboth{THE INHOMOGENEOUS LIPPMANN-SCHWINGER EQUATION}{G. GIORGI, M. BRIGNONE, R. ARAMINI AND M. PIANA}

\section{Introduction}
Acoustic, elastic or electromagnetic scattering is a physical phenomenon where an incident wave is scattered by an inhomogeneity (or an obstacle) and the
total field at any point is represented by the sum of the incident and the scattered field. From a mathematical viewpoint, in the direct problem the physical parameters and the geometry of the inhomogeneity are given and the unknown is represented by the scattered field. In the inverse scattering problem, one aims at recovering the shape and the physical properties of the object of interest from measurements of the scattered field. The direct scattering problem is in general well-posed in the sense of Hadamard and therefore its approximate solution can be determined by means of stable numerical methods. On the contrary, the inverse scattering problem is ill-posed in the sense of Hadamard (specifically, the unknown physical parameters of the scatterer are mapped onto the measured scattered field by a compact operator), and its solution must be addressed by means of some
regularizing approach \cite{cokr98}.

Most inverse scattering methods belong to two different sets of algorithms: the family of qualitative approaches and the family of quantitative approaches. Qualitative methods \cite{caco06} require the regularized solution of a linear integral equation of the first kind, parameterized over
a grid of sampling points covering the investigation domain. The Euclidean norm of such regularized solution behaves as an indicator function of the unknown scatterer, since it is bounded when the sampling point is inside the target, grows up if the point approaches its edge and can
be made arbitrarily large when it is taken outside. The advantages of qualitative methods are that they are fast and need very few a priori
information to work. However, the informative content of their output is rather poor, since they are essentially visualization techniques
and do not provide quantitative reconstructions of the inhomogeneity.

On the other hand, quantitative inverse scattering methods \cite{bako04,cokr98} are, in general, iterative schemes that minimize an appropriate
functional, starting from an initialization mask for the refractive index of the scatterer, and
then reconstruct its point values after an optimized number of iterations. In principle, they can provide all the required information on the problem,
although by means of a notable computational effort. However, they can suffer from local minima problems, since the number of unknowns is typically larger than the data amount at disposal in scattering experiments, or because often the initialization of the method is not precise enough for a proper convergence to the solution. The effectiveness of a quantitative method can be notably increased by incorporating some \textit{a priori}
information (when available) on the scatterer into the minimization process. Although the most traditional approach is to use this knowledge
for a better initialization of the algorithm, an improvement in its performance can be achieved when such information is integrated into each step
of the iterative scheme, e.g., by decreasing the complexity of the target \cite{giletal09}, or by adding appropriate constraints to the minimization
technique \cite{vdbbrab99}.

Recent developments in inverse scattering are concerned with the formulation of hybrid methods,
i.e., methods merging different techniques in order to integrate and optimize the different kinds of information they can provide
(see e.g. \cite{zioibrido08,brcopi07,krse07}).
In particular, \textit{a priori} knowledge on the scatterer and/or qualitative techniques are typically utilized to
improve the performance of some quantitative method, e.g., by reducing
the number of unknowns, or by providing a more precise initialization.
In this paper, we rely on a rather general strategy for the formulation of hybrid techniques, based on the following three steps:
\begin{enumerate}
\item the \textit{a priori} information on the physical characteristics of the (possibly inhomogeneous) background is coded into the
      corresponding Green's function;
\item a segmentation between the background and the scatterer is realized by applying a qualitative method, thus reducing the size of the investigation domain, i.e., the number of unknowns in the following step;
\item a quantitative method is applied only in the region highlighted by the qualitative method, in order to reconstruct the point values of the refractive index in this region.
\end{enumerate}

Although this scheme is of general applicability, in this paper we present a specific implementation: more precisely, the Green's function of
the (inhomogeneous) background is numerically computed by means of the method of moments \cite{harrington,ri65}, which solves the forward
scattering problem. Then the linear sampling method \cite{copipo97} is applied in its near-field formulation for inhomogeneous backgrounds
\cite{cafaha06,como98} in order to visualize the scatterer of interest, and a procedure based on active contours \cite{arbrcopi08,chve01} is used for extracting the (approximate) shape of its support.
Finally, the inverse scattering problem is solved by generalizing and applying the contrast source inversion method \cite{giletal09,vdbkl97,vdbbrab99} to what we call the `inhomogeneous Lippmann-Schwinger equation', since
it is a generalization of the integral formulation of the direct scattering problem to the case of an inhomogeneous background:
such equation is deduced in our paper as a preliminary but essential tool that allows encoding the
\textit{a priori} information on the background into each step of the iterative reconstruction procedure.
Note that the initial qualitative approach allows computing the point values of the refractive index only in the region identified by the linear
sampling method, thus reducing the number of unknowns to be determined by the quantitative method.

The plan of the paper is as follows. 
In Section \ref{ziodifferenziale} we introduce the differential formulation of the direct scattering 
problem and recall some known results. Section \ref{zioinomogeneo} addresses the same problem from an integral perspective,  
%set up the scattering problem 
%\ref{zioinomogeneo} we set up the scattering problem from both a differential and integral perspective, 
and as a result the inhomogeneous Lippmann-Schwinger equation is derived from the direct differential formulation.
In Section \ref{zioibrido} we recall the key ideas of the linear sampling method in the near-field case, 
and adapt the contrast source inversion method to the inhomogeneous Lippmann-Schwinger equation. In Section \ref{zionumerico} 
some numerical examples are described, also including an application to breast cancer detection by using microwaves. 
Finally, our conclusions are offered in Section \ref{zioconcluso}.

\section{The scattering problem: differential formulation}\label{ziodifferenziale}

We consider a rather general two-dimensional and time-harmonic scattering problem, whereby the background medium is inhomogeneous and 
described by a piecewise continuously differentiable refractive index \cite{cokr98,como98}
\be\label{indice}
n_b(x)=\frac{1}{\varepsilon_0}\left[\varepsilon(x)+\mathrm{i}\frac{\sigma(x)}{\omega}\right].
\ee
In (\ref{indice}), $x=(x_1,x_2)$ is a generic point of $\R^2$, $\varepsilon_0>0$ is the dielectric permittivity of vacuum,
$\varepsilon(x)\geq 1$ and $\sigma(x)\geq 0$ are the point values of the dielectric permittivity and conductivity of the medium,
$\omega$ is the angular frequency,
and $\mathrm{i}=\sqrt{-1}$. More precisely (see e.g. fig. \ref{schema}), we assume that there exists a finite number $N+1$ of 
open and connected
$C^2$-domains $\Omega_i\subset\R^2$, with $i=0,\ldots,N$, such that 1) $\Omega_i\cap\Omega_j=\emptyset$ for $i\neq j$;
2) $\R^2=\cup_{i=0}^{N} \bar{\Omega}_i$; 3) $\Omega_i$ is bounded for each $i\neq 0$; 4) $n_b|_{\Omega_i} \in C^1(\bar{\Omega}_i)$ for all
$i=0,\ldots,N$; 5) $n_b(x)=\tilde{n}_0\in\C$ for $x\in\Omega_0$, with $\mathrm{Im}\{\tilde{n}_0\}\geq 0$; 6) there exists a subset $J_0$ of 
the finite set $\{1,2,\ldots,N\}$ such that $\tilde{\Omega}_0:=\mathrm{int}\left\{x\in\R^2 : n_b(x)=\tilde{n}_0 \right\}=
\Omega_0\cup\left(\cup_{j\in J_0}\Omega_j\right)$ (in particular, $J_0=\emptyset$ if and only if $\tilde{\Omega}_0=\Omega_0$).
Finally, we assume that the magnetic permeability is constant on all $\R^2$.
Note that a) the domains $\Omega_i$ do not need to be simply connected: e.g., in fig. \ref{schema}, $\Omega_3$ has $\Omega_4$ 
as its hole and, in turn, $\Omega_1$ has holes 
corresponding to $\Omega_2$ and $\Omega_3\cup\Omega_4$; b) the boundaries 
$\partial\Omega_i$ may be either curves where a discontinuity of $n_b$ occurs, or boundaries of virtual 
domains $\Omega_i$ immersed in a homogeneous region, where to host a scatterer that will be introduced later:
this trick allows some notational simplifications.
%Finally, we define the open and possibly non-connected set
%$\tilde{\Omega}_0:=\mathrm{int}\left\{x\in\R^2: n_b(x)=\tilde{n}_0 \right\}\supset\Omega_0$.
%and assume that it is also a $C^2$-domain.

\begin{figure}[H]
\begin{center}
%\begin{tabular}{c}
\epsfig{file=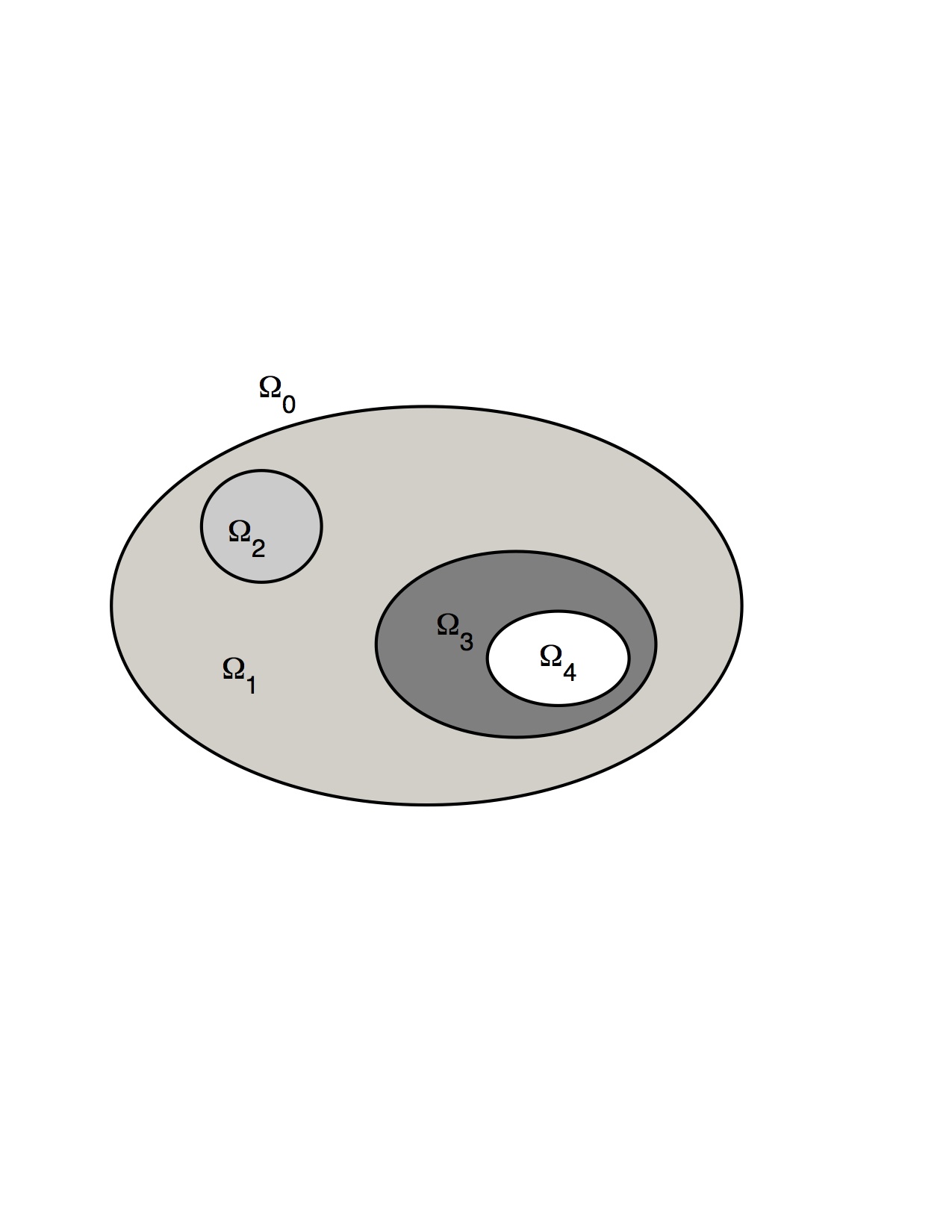, width=5.6cm}
%\end{tabular}
\caption{\label{schema}Scheme of the reference background for the scattering problem. In this case, $\tilde{\Omega}_0=\Omega_0\cup\Omega_4$.}
\end{center}
\end{figure}

Next, we define the Green's function $G(x;y)$ of the background as the radiating solution of the equation \cite{como98}
\be\label{graun}
\Delta_x G(x;y) + k^2 n_b(x)\, G(x;y)=-\delta(x-y)\ \ \ \mbox{for}\ \ x\in\R^2,   %\setminus\{y\},
\ee
where $k=\omega/c$ is the wave number in vacuum, $c$ being the speed of light in free space.
%the wave number $k>0$ in the outmost region $\Omega_0$ is defined by $k_0:=\omega/c_0$, $c_0$ being the speed of light in $\Omega_0$.
%by $k^2=\varepsilon_0\mu_0\omega^2$, $\mu_0>0$ and $\omega>0$ being respectively
%the magnetic permeability (assumed to be constant in $\R^2$) and the angular frequency of the scattering experiment.
The existence and uniqueness of $G(x;y)$ (at least for $y\in\tilde{\Omega}_0$)
%$y\in\tilde{\Omega}_0$, with $\tilde{\Omega}_0:=\mathrm{int}\,\{x\in\R^2 : n_b(x)=1\}$)
can be proved as in \cite{cokr98,como98}: in fact, if we denote the Green's function of the homogeneous medium in $\tilde{\Omega}_0$
by $\Phi (x;y):=\frac{\mathrm{i}}{4} H^{(1)}_0(\tilde{k}_0\,|x-y|)$, where $H^{(1)}_0$ is the Hankel function of first kind of order zero, 
$\tilde{k}_0:=k\sqrt{\tilde{n}_0}$ and $\mathrm{Im}\{\sqrt{\tilde{n}_0}\}\geq 0$, we can write
\be\label{alluvione}
G(x;y)=\Phi(x;y)+u^s_b(x;y),
\ee
where $u^s_b(x;y)$ is the perturbation to $\Phi(x;y)$ due to the inhomogeneous medium in $\R^2\setminus\tilde{\Omega}_0$. 

We point out that \cite{como98} $G(\cdot;y)\in C^1(\R^2\setminus\{y\})$, i.e., the discontinuities of $n_b$ across some of 
the boundaries $\partial\Omega_i$ only affect the smoothness of second (and higher order) derivatives on the boundaries themselves. 
Actually, eq. (\ref{graun}) is to be understood as a set of $N+1$ equations, one for each domain $\Omega_i$, linked by 
proper transmission conditions at the boundaries $\partial \Omega_i$.
These conditions are imposed by physics, which states the continuity of the tangential components of both the total electric field $E$ 
and the magnetic field $H$: this is equivalent to the continuity of the vector fields $\nu\times E$ and $\nu\times H$, where `$\times$' 
denotes the vector product and $\nu$ is the unit normal at a point of $\partial\Omega_i$.
In our two-dimensional setting, Cartesian axes are chosen so that $G$ is the non-zero component of the total electric field
%we can choose Cartesian orthogonal axes in such a way that 
$E=(0,0,G)$, which vibrates perpendicularly to the scattering plane (i.e., the electromagnetic field is assumed to be TM-polarized). 
Moreover, from the time-harmonic Maxwell 
equation \cite{cokr98} $\mathrm{curl} E - \mathrm{i}k H=0$, it follows that $H=\frac{1}{\mathrm{i}k}(\partial_2 G, -\partial_1 G,0)$,
having denoted by $\partial_j$ the partial derivative with respect to the variable $x_j$, for $j=1,2$. Finally, in the same reference
system, the normal $\nu$ can be written as $(\nu_1,\nu_2,0)$: accordingly, $\nu\times E=(\nu_2\, G,-\nu_1\, G,0)$ and 
$\nu\times H=(0,0,-\nu_1\,\partial_1 G-\nu_2\,\partial_2 G)$. Then, the continuity
of $\nu\times E$ and $\nu\times H$ corresponds to the continuity of $G$ and its normal derivative
$\nu\cdot\nabla G=\partial_\nu G$.

%from the time-harmonic Maxwell equations \cite{cokr98}
%\be
%\mathrm{curl} E -\mathrm{i}k H=0,\ \ \mathrm{curl} H + \mathrm{i} k n(x) E=0
%\ee
%and the assumption that the magnetic permeability is constant on all 
%$\R^2$, it is easily seen that this amounts to the continuity of
%$u$ and $\frac{\partial u}{\partial\nu}$, where $\nu$ is the normal at a point of $\partial\Omega_i$.
% Then, $G(x;y)$ 
% satisfies the uniquely solvable Lippmann-Schwinger equation
% \be\label{figalli}
% G(x;y)=\Phi(x;y)-\tilde{k}_0^2\int_{\R^2}\Phi(x;z) m_b(z)G(z;y)\,dz\ \ \ \ \mathrm{for}\ x\in\R^2,
% \ee
% with $m_b(x):=[\tilde{n}_0-n_b(x)]/\tilde{n}_0$ and $y\in\tilde{\Omega}_0$. In particular, $G(\cdot;y)\in C^1(\R^2\setminus\{y\})$, since
% discontinuities of $m_b(x)$ across some of the boundaries $\partial\Omega_i$ only affect the smoothness of second (and higher order) 
% derivatives of $G(\cdot;y)$ \cite{cokr98,como98,kellogg}.  

In our scattering problem, the scatterer is assumed to take up the spatial region $D=\cup_{i=1}^M\Omega_i$, with $1\leq M\leq N$; accordingly, 
the whole propagation medium is described by a refractive index $n(x)$ such that, in general, $n(x)\neq n_b(x)$ for $x\in \Omega_i$, 
with $i=1,\ldots,M$. In any case, we still require that $n|_{\Omega_i}\in C^1(\bar{\Omega}_i)$ for all $i=0,\ldots,N$.
%Note that, since the background is physically well-defined only outside the scatterer (while it is somewhat arbitrary inside),
%for our purposes it is not restrictive to assume $n_b(x)=1$ for $x\in D$, provided that $n(x)\neq 1$ for $x\in D$.
Moreover, for a unit point source placed at $x_0\in\tilde{\Omega}_0\setminus \overline{(\tilde{\Omega}_0\cap D)}$,
%\tilde{\Omega}_0:=\mathrm{int}\,\{x\in\R^2 : n(x)=1\}$,
%$x_0\in\R^2\setminus\left[\bar{D}\cup\left(\cup_{i=0}^N\partial\Omega_i\right)\right]$,
we denote by $u(x;x_0)$ the non-zero component of the total electric field $E=(0,0,u)$, which, as before, vibrates perpendicularly
to the scattering plane. 
%(i.e., the electromagnetic field is assumed to be TM-polarized).
Then, the differential form of the scattering problem we are interested in is \cite{como98}
\begin{equation}\label{scatteringnear}
%\hspace{-2cm}
\left\{
\begin{array}{lclc}
\displaystyle\Delta_x u(x;x_0) +k^2 n(x)\, u(x;x_0)=-\delta(x-x_0) & & {\mathrm{for}}\ \ x\in\R^2 & {\mathrm{\ \ \ \ (a)}} \\ [2mm]
\displaystyle u(x;x_0)=u^i(x;x_0) + u^s(x;x_0) & & \phantom{{\mathrm{for}}\ \ x\in\R^2} & {\mathrm{\ \ \ \ (b)}} \\  [2mm]
\displaystyle \lim_{r\rightarrow\infty} \left[\sqrt{r}\left(\frac{\partial u^s}{\partial r}-{\mathrm{i}}\,
\tilde{k}_0 u^s\right)\right]=0,
& & & {\mathrm{\ \ \ \ (c)}}
\end{array}
\right.
\end{equation}
where $u^i(x;x_0)=G(x;x_0)$ is the incident field, $u^s(x;x_0)$ is the scattered field and (\ref{scatteringnear})(c) is the Sommerfeld 
radiation condition, which holds uniformly in all directions $\hat{x}=x/|x|$. If $\mathrm{Im}\{\tilde{n}_0\}>0$, such a condition can be 
relaxed \cite{como98} by only requiring the boundedness of $u^s(\cdot;x_0)$ in $\R^2$.

Note that, from eq. (\ref{alluvione}) and the identification $u^i(x;x_0)=G(x;x_0)$, 
we can rewrite eq. (\ref{scatteringnear})(b) as $u(x;x_0)= \Phi(x;x_0)+[u^s_b(x;x_0)+ u^s(x;x_0)]$: then, as before, the existence and 
uniqueness of a solution $u(\cdot;x_0)\in C^1(\R^2\setminus\{x_0\})$ to problem (\ref{scatteringnear}) can be proved as
in \cite{como98}, and eq. (\ref{scatteringnear})(a) is again to be regarded as a set of $N+1$ equations, one for each domain 
$\Omega_i$, with the above transmission conditions at the boundaries $\partial \Omega_i$.

Moreover, since both the incident field $u^i(\cdot;x_0)$ and the total field $u(\cdot;x_0)$ are in $C^1(\R^2\setminus\{x_0\})$, 
also the scattered field $u^s(\cdot;x_0)$, expressed by (\ref{scatteringnear})(b) as the difference $u(\cdot;x_0)-u^i(\cdot;x_0)$, is in 
$C^1(\R^2\setminus\{x_0\})$. Actually, it is not difficult to establish further regularity properties of $u^s(\cdot;x_0)$, i.e.,
$u^s(\cdot;x_0)\in C^1(\R^2)$ and $\Delta u^s(\cdot;x_0)\in C^1\left(\bar{\Omega}_i\right)$ for all $i=0,\ldots,N$. To this end,
we observe that a simple algebraic manipulation of eq. (\ref{scatteringnear})(a) yields
\be
\Delta_x u(x;x_0) + k^2 n_b(x) u(x;x_0)=-\delta(x-x_0) + k^2 \tilde{n}_0 \frac{n_b(x)-n(x)}{\tilde{n}_0} u(x;x_0),
\ee
i.e., by setting $m(x):=[n_b(x)-n(x)]/\tilde{n}_0$ and remembering that $\tilde{k}_0=k\sqrt{\tilde{n}_0}$,
\be\label{ruggero}
\Delta_x u(x;x_0) + k^2 n_b(x) u(x;x_0)=-\delta(x-x_0) + \tilde{k}_0^2 m(x) u(x;x_0).
\ee
By using equations (\ref{graun}) and (\ref{scatteringnear})(b), as well as the above identification $u^i(x;x_0)=G(x;x_0)$, 
eq. (\ref{ruggero}) is easily seen to become
\be\label{errani}
\Delta_x u^s(x;x_0) + k^2 n_b(x) u^s(x;x_0)=-\left[-\tilde{k}_0^2 m(x) u(x;x_0)\right].
\ee
The structure of eq. (\ref{errani}), which parallels that of eq. (\ref{graun}), allows regarding $u^s(\cdot;x_0)$ as the third 
component of the electric field $E^s=(0,0,u^s)$ radiated by the equivalent source 
$-\tilde{k}_0^2 m(x) u(x;x_0)$ in the background medium described by the refractive index $n_b$.
We now note that this equivalent source is compactly supported, since $\mathrm{supp}\, m$ coincides with the support $\bar{D}$ 
of the scatterer, which is compact by assumption. On the other hand, by hypothesis, 
we have that $x_0\in\tilde{\Omega}_0\setminus \overline{(\tilde{\Omega}_0\cap D)}$, i.e., the point $x_0$ is at a finite distance 
from the equivalent source: accordingly, there exists an open neighbourhood $U_{x_0}\subset \tilde{\Omega}_0$ of $x_0$ such 
that 1) $U_{x_0}\cap \mathrm{supp}\, m=\emptyset$; 2) both the electric field $E^s$ and the corresponding magnetic field 
$H^s=\frac{1}{\mathrm{i}k}(\partial_2 u^s,-\partial_1 u^s,0)$ are well-defined and bounded in $U_{x_0}$. Even more, from the
transmission conditions imposed by physics and previously recalled, we can conclude that both $u^s(\cdot;x_0)$ and 
$\partial_{\nu}u^s(\cdot;x_0)$ are in $C^0(\bar{U}_{x_0})$. But here the unit normal is arbitrary, since no physical interface, i.e., no
discontinuity of the refractive index occurs inside $U_{x_0}$, where in fact $n_b(x)=\tilde{n}_0$: then, in particular, 
$\partial_j u^s(\cdot;x_0)\in C^0(\bar{U}_{x_0})$ for $j=1,2$. Accordingly, we have that $u^s(\cdot;x_0)\in C^1(\bar{U}_{x_0})$.
On the other hand, we know that $u^s(\cdot;x_0)\in C^1(\R^2\setminus
\{x_0\})$: we then conclude that $u^s(\cdot;x_0)\in C^1(\R^2)$.

Finally, from eq. (\ref{errani}) we find
\be\label{ridicola}
\Delta_x u^s(x;x_0) =- k^2 n_b(x) u^s(x;x_0)+\tilde{k}_0^2 m(x) u(x;x_0).
\ee
Since, by hypothesis, $n_b,m\in C^1(\bar{\Omega}_i)$ for all $i=0,\ldots,N$, while $u^s(\cdot;x_0)\in C^1(\R^2)$ and 
$u(\cdot;x_0)\in C^1(\R^2\setminus\{x_0\})$, from (\ref{ridicola}) we have $\Delta u^s(\cdot;x_0)\in C^1(\bar{\Omega}_i)$: indeed, the
singularity of $u(x;x_0)$ for $x=x_0$ is cancelled out by $m(x)$, since $x_0\notin \mathrm{supp}\, m$ by assumption, while $u^s(\cdot;x_0)$ 
is bounded at infinity (actually, it is bounded in all $\R^2$ \cite{como98}).

Of course, as a particular case (i.e., $n(x)=n_b(x)$ $\forall x\in\R^2$) of the above discussion, 
the same regularity properties also hold for the field 
$u^s_b(\cdot;y)$ introduced in (\ref{alluvione}), i.e., 
$u^s_b(\cdot;y)\in C^1(\R^2)$ and $\Delta u^s_b(\cdot;y)\in C^1\left(\bar{\Omega}_i\right)$ for all $i=0,\ldots,N$.

% These conditions are expressed by the continuity of the tangential components of both the total electric field $E$ and the magnetic field 
% $H$: from the time-harmonic Maxwell equations and the assumption that the magnetic permeability is constant on all $\R^2$,
% it is easily seen that this amounts to the continuity of
% $u$ and $\frac{\partial u}{\partial\nu}$, where $\nu$ is the normal at a point of $\partial\Omega_i$. Since the same conditions hold 
% for the incident field (which is the total field in absence of the scatterer), we can look for a solution $u^s(\cdot;x_0)$ to problem 
% (\ref{scatteringnear}) such that $u^s(\cdot;x_0)\in C^1\left(\R^2\right)$
% %$\Delta u^s(\cdot;x_0)\in C\left(\bar{\Omega}_0\setminus\{x_0\}\right)$
% and $\Delta u^s(\cdot;x_0)\in C^0\left(\bar{\Omega}_i\right)$ for all $i=0,\ldots,N$.
% %having adopted the shorthand notation $\Omega_i\setminus\{x_0\}$ for $\Omega_i\setminus\left(\Omega_i\cap\{x_0\}\right)$.
% Again, the existence and uniqueness of such a solution can be proved by paralleling arguments contained in \cite{como98}.

\section{The inhomogeneous Lippmann-Schwinger equation}\label{zioinomogeneo}

The goal of the present section is to derive an integral equation for the scattered field $u^s(\cdot;x_0)$ from the differential formulation
(\ref{scatteringnear}) of the scattering problem. 
Observing the structure of eq. (\ref{errani}), the natural candidate for an integral 
formulation is the `generalized' Lippmann-Schwinger equation
\be\label{lippus}
u^s(x;x_0)=-\tilde{k}_0^2\int_{\R^2}G(x;y)m(y)u(y;x_0)\,dy.
\ee
%where $m(x):=[n_b(x)-n(x)]/\tilde{n}_0$ (note that $\mathrm{supp}\, m$ coincides with the support $\bar{D}$ of the scatterer, and is 
%therefore compact).
In fact, an exact equivalence between differential and integral formulation is proved in \cite{cokr98} for the three-dimensional acoustic 
and electromagnetic cases under strong regularity assumptions on the refractive index: the background is assumed to be homogeneous 
(i.e., $n_b(x)=\tilde{n}_0\in\C$ for all $x\in\R^3$), and $n$ must be continuously differentiable on the whole space 
(however, see 
\cite{ki11} for a variational approach notably relaxing the latter requirement). 
In particular, under such assumptions, this proof can provide a rigorous justification of the following procedure: compute 
$\Delta_x u^s(x;x_0)$ by interchanging the Laplacian operator with the integral in (\ref{lippus}), then replace
\be\label{peccato}
\Delta_x \Phi(x;y)=-\tilde{k}_0^2 \Phi(x;y)-\delta(x-y)
\ee
(which is the analogous of eq. (\ref{graun}) in the homogeneous case $G(x;y)=\Phi(x;y)$) into eq. (\ref{lippus}) to obtain
\be\label{incredibile}
\Delta_x u^s(x;x_0) + \tilde{k}_0^2 u^s(x;x_0)=-\left[-\tilde{k}_0^2 m(x)u(x;x_0)\right].
\ee
Eq. (\ref{incredibile}) suggests regarding the term in square brackets at its right-hand side as an equivalent, although unknown, 
source radiating the field $u^s(x;x_0)$ in the homogeneous background, which is consistent with the physical 
interpretation of (\ref{lippus}). In any case, by means of eq.s (\ref{peccato}) and (\ref{incredibile}), it is easily verified that $u^s$, 
as given by (\ref{lippus}), solves problem (\ref{scatteringnear}).

However, as correctly pointed out in \cite{ma03}, even a mere discontinuity of the physical parameters at the interface $\partial D$ between the
scatterer and the (homogeneous) background suffices to invalidate, in general, the previous procedure: as a consequence, even with a homogeneous
background, the integral formulation becomes more complicated (owing to the occurrence of boundary terms on $\partial D$) and its equivalence with
the differential formulation is more difficult to prove.

To our knowledge, in the case of an inhomogeneous background, results concerning the equivalence between the differential and integral formulation
of a scattering problem are not available. In this section, we shall limit ourselves to deriving the integral equation
(\ref{lippus}) (for $x\in\tilde{\Omega}_0$) from the differential problem (\ref{scatteringnear}):
indeed, a thorough discussion of the equivalence between (\ref{lippus}) and (\ref{scatteringnear}) would be very
technical and beyond the framework of this paper.
However, it is likely that an exact equivalence actually holds,
%since the traditional approach using equivalent unknown sources should still
%work for our two-dimensional electromagnetic scattering problem. Indeed,
since, as recalled above, the scattered field $u^s$ is in $C^1(\R^2)$: then,
the boundary contributions on each $\partial \Omega_i$ that would appear in (\ref{lippus}) from Green's second theorem \cite{cokr98}
applied in each domain $\Omega_i$ cancel out, as detailed in the following Lemma \ref{salone} and Theorem \ref{teorema} for the case
$x\in\tilde{\Omega}_0$.
%However, it is likely that an exact equivalence actually holds, since the traditional approach using equivalent unknown sources should still
%work for our two-dimensional electromagnetic scattering problem: this is due to the fact, recalled above, that the scattered field $u^s$ is in
%$C^1(\R^2)$, so that the boundary contributions on each $\partial \Omega_i$ that would appear in (\ref{lippus}) from Green's second theorem
%\cite{cokr98} applied in each domain $\Omega_i$ cancel out.

Finally, it is worth noting that, even in two dimensions, the previous argument fails \cite{ma03} whenever the magnetic permeability is discontinuous across $\partial\Omega_i$. This is analogous to what happens in the two or three-dimensional acoustic case \cite{ma03}, where the pressure field is continuous across the discontinuities of the fluid density, but not continuously differentiable.

%Before stating Theorem \ref{teorema}, we need the following lemma.
\begin{lemma}\label{salone}
Let $\Omega_i$ (with $i=0,\ldots,N$), $\tilde{\Omega}_0$, $n_b(x)$, $G(x;y)$ be as above,
let $B_R:=\{x\in\R^2 : |x|< R\}$ be such that $B_R\supset \cup_{i=1}^N\bar{\Omega}_i$, and define
$\Omega_{N+1}:=B_R\setminus \cup_{i=1}^N\bar{\Omega}_i$. Moreover, let $w\in C^1 (\bar{B}_R)$ be such that
$\Delta w\in C^0(\bar{\Omega}_i)$ for all $i=1,\ldots,N+1$.
%exists as an element of $L^1(B_R)$.
Then, the following generalization of Green's formula holds:
\begin{align}\label{generalizz}
w(x)=&\int_{\partial B_R}\left[\frac{\partial w}{\partial\nu}(y) G(y;x)-w(y)\frac{\partial G(y;x)}{\partial\nu(y)}\right]ds(y)+\\
 &-\int_{B_R}\left[\Delta w(y)+k^2 n_b(y) w(y)\right] G(y;x)\,dy\ \ \ \ \ \ \ \ \ \ \forall x\in B_R\cap\tilde{\Omega}_{0}.\nonumber
\end{align}
\end{lemma}

\begin{proof}
Consider an arbitrary point $x\in\Omega_R:=B_R\cap\tilde{\Omega}_{0}$: since $\Omega_R$ is open, there exist $\rho >0$ and a ball
$B(x;\rho):=\{y\in\R^2 : |y-x|<\rho\}$ such that $\bar{B}(x;\rho)\subset\Omega_{R}$. Moreover,
%by reciprocity \cite{como98,landau8},
$G(\cdot;x)$ solves a particular case of the differential problem (\ref{scatteringnear}), with the identifications $x_0=x$, 
$u(\cdot;x_0)=G(\cdot;x)$, $u^i(\cdot;x_0)=\Phi(\cdot;x)$ and $n=n_b$: accordingly, 
remembering eq. (\ref{alluvione}) and the regularity 
properties stated in the previous section, we have that 
%\cite{como98} 
$G(\cdot;x)\in C^1\left(\bar{B}_R\setminus\{x\}\right)$,
$\Delta G(\cdot;x)\in C^1\left(\bar{\Omega}_{R}\setminus\{x\}\right)$ and
$\Delta G(\cdot;x)\in C^1(\bar{\Omega}_i)$ for all $i\in\{1,\ldots,N\}\setminus J_0$,
where the index set $J_0$ has been defined in assumption no. 6) soon below eq. (\ref{indice}).
%and $\Delta G(\cdot;x)\in C^0\left(\bar{\Omega}_{N+1}\setminus\{x\}\right)$.
Then, given the regularity properties assumed for $w$,
we can apply the usual Green's second theorem \cite{cokr98,kellogg} in the domain $\Omega_{R}\setminus\bar{B}(x;\rho)$, i.e.,
\begin{align}\label{allerta1}
&\int_{\Omega_{R}\setminus\bar{B}(x;\rho)}\left[G(y;x)\Delta w(y)-w(y)\Delta_y G(y;x)\right]dy=\\
&=\int_{\partial\Omega_{R}\cup\partial{B}(x;\rho)}\left[G(y;x)\frac{\partial w(y)}{\partial \nu} -
w(y)\frac{\partial G(y;x)}{\partial\nu(y)}\right]ds(y),\nonumber
\end{align}
as well as in any other domain $\Omega_i$, with $i\in\{1,\ldots,N\}\setminus J_0$, i.e.,
\begin{align}\label{allerta2}
&\int_{\Omega_{i}}\left[G(y;x)\Delta w(y)-w(y)\Delta_y G(y;x)\right]dy=\\
&=\int_{\partial\Omega_{i}}\left[G(y;x)\frac{\partial w(y)}{\partial \nu} -
w(y)\frac{\partial G(y;x)}{\partial\nu(y)}\right]ds(y).\nonumber
\end{align}
In eq.s (\ref{allerta1}) and (\ref{allerta2}), it is understood that the unit normal is chosen as outward with respect to each domain. We now sum eq.
(\ref{allerta1}) with all the equations (\ref{allerta2}) obtained for $i\in\{1,\ldots,N\}\setminus J_0$: note that, except for
$\partial B_R$ and $\partial{B}(x;\rho)$, all the boundary integrals are taken twice, with opposite orientation of the unit normal. Since the integrand functions are continuous on the boundaries, these integrals cancel out. Moreover, by eq. (\ref{graun}), we can substitute $\Delta_y G(y;x)=-k^2 n_b(y) G(y;x)$ for $y\neq x$ into
(\ref{allerta1}) and (\ref{allerta2}). As a result, we find
\begin{align}\label{allerta3}
&\int_{B_R\setminus\bar{B}(x;\rho)}\left[\Delta w(y)+k^2 n_b(y) w(y)\right]G(y;x)dy=\\
&=\int_{\partial B_R\cup\partial{B}(x;\rho)}\left[G(y;x)\frac{\partial w(y)}{\partial \nu} -
w(y)\frac{\partial G(y;x)}{\partial\nu(y)}\right]ds(y).\nonumber
\end{align}
We now focus on the integral over $\partial{B}(x;\rho)$, say $I_{\partial{B}(x;\rho)}$, at the right-hand side of (\ref{allerta3}): remembering eq. (\ref{alluvione}), we have
\begin{align}\label{pioggia}
I_{\partial{B}(x;\rho)}=&\int_{\partial{B}(x;\rho)}\left[\Phi(y;x)\frac{\partial w(y)}{\partial\nu}-w(y)\frac{\partial\Phi(y;x)}{\partial\nu(y)}\right]ds(y)+\\ \nonumber
&+\int_{\partial{B}(x;\rho)}\left[u_b^s(y;x)\frac{\partial w(y)}{\partial\nu}-w(y)\frac{\partial u_b^s(y;x)}{\partial\nu(y)}\right]ds(y).
\end{align}
Now, the second integral in (\ref{pioggia}) vanishes as $\rho\rightarrow 0$, since the integrand is bounded and the measure of the integration domain tends to zero. As regards the first integral in (\ref{pioggia}), we recall \cite{caco06} the following asymptotic behaviors
\begin{align}
\Phi(y;x) & =\frac{1}{2\pi}\ln\frac{1}{|x-y|}+O(1) & \mbox{as}\ \ |x-y|\rightarrow 0,\label{piu}\\
\frac{\partial \Phi(y;x)}{\partial\nu(y)} & =\frac{1}{2\pi}\frac{1}{|x-y|}+O(|x-y|\ln|x-y|) & \mbox{as}\ \ |x-y|\rightarrow 0.\label{meno}
\end{align}
By using (\ref{piu}) and (\ref{meno}), the integral mean value theorem applied to the first integral in (\ref{pioggia}) easily shows that the latter
tends to $-w(x)$ as $\rho\rightarrow 0$. Then, thesis (\ref{generalizz}) is obtained by taking $\rho\rightarrow 0$ in eq. (\ref{allerta3}): indeed, remembering eq. (\ref{alluvione}) and the regularity of $u_b^s(y;x)$, the singularity of $G(y;x)$ for $y\rightarrow x$ is only due to $\Phi(y;x)$, i.e., it is weak and then the integral on $B_R$ converges.
\end{proof}

\begin{theorem}\label{teorema}
Let $\Omega_i$ (with $i=0,\ldots,N+1$), $\tilde{\Omega}_0$, $D$, $n_b(x)$, $n(x)$, $m(x)$, $G(x;y)$ be as above, and let
$x_0\in\tilde{\Omega}_0\setminus \overline{(\tilde{\Omega}_0\cap D)}$ be as above.
%Then,
%(i) if $u^s(\cdot;x_0)\in C(\R^2)$ is a solution of the integral equation (\ref{lippus}), then $u^s(\cdot;x_0)\in C^1(\R^2)$,
%with $u^s(\cdot;x_0)\in C^2\left(\Omega_i\right)$ for all $i=0,\ldots,N$, and $u^s(\cdot;x_0)$ is a solution of the differential problem
%(\ref{scatteringnear});
%(ii) conversely, if
If $u^s(\cdot;x_0)\in C^1(\R^2)$, with $\Delta u^s(\cdot;x_0)\in C^1\left(\bar{\Omega}_i\right)$ for all $i=0,\ldots,N$, is a solution of the differential problem (\ref{scatteringnear}), then $u^s(\cdot;x_0)$ solves the integral equation (\ref{lippus}) for $x\in\tilde{\Omega}_0$.
\end{theorem}

\begin{proof}
Consider an arbitrary point $x\in\tilde{\Omega}_0$. Let $B_R:=\{x\in\R^2 : |x|<R\}$ be an open disk with exterior unit normal $\nu$ and radius
$R$ large enough, so that $B_R\supset\left[\cup_{i=1}^N\bar{\Omega}_i\cup\{x\}\right]$.
By hypothesis, $u^s(\cdot;x_0)$ is regular enough in the domain $B_R$ to be represented by means of the generalized Green's formula (\ref{generalizz}) in $\Omega_{R}:=B_R\cap\tilde{\Omega}_0$, i.e.,
\begin{align}
u^s(x;x_0) = &\int_{\partial B_R}\left\{\frac{\partial u^s(y;x_0)}{\partial\nu(y)}G(y;x)-u^s(y;x_0)
\frac{\partial G(y;x)}{\partial\nu(y)}\right\} ds(y) + \label{quiz}\\
& - \int_{B_R} \left\{\Delta_y u^s(y;x_0)+k^2 n_b(y) u^s(y;x_0)\right\} G(y;x) dy,\ \ \ \ \ \   x \in \Omega_{R}.\nonumber
\end{align}
First we prove that the integral on $\partial B_R$ in (\ref{quiz}) is zero. To this end, consider $B_r:=\{x\in\R^2 : |x|<r\}$ such that $r>R$, and apply Green's second theorem \cite{cokr98} in the domain $B_r\setminus B_R$: by choosing the unit normal $\nu$ as outgoing from both $B_R$ and $B_r$, and observing that both the fields $u^s(\cdot;x_0)$ and $G(\cdot;x)$ verify the same Helmholtz equation in $B_r\setminus B_R\subset\Omega_0$ (where $n_b(x)=n(x)=\tilde{n}_0$), we find
\begin{align}\label{monti}
&\int_{\partial B_R}\left\{\frac{\partial u^s(y;x_0)}{\partial\nu(y)}G(y;x)-u^s(y;x_0)
\frac{\partial G(y;x)}{\partial\nu(y)}\right\} ds(y)=\\
& =\int_{\partial B_r}\left\{\frac{\partial u^s(y;x_0)}{\partial\nu(y)}G(y;x)-u^s(y;x_0)
\frac{\partial G(y;x)}{\partial\nu(y)}\right\} ds(y).\nonumber
\end{align}
We now recall \cite{cokr98} that any radiating solution $v$ of the Helmholtz equation (with generic wave number $k>0$) outside a disk in $\R^2$
has the following asymptotic behavior
\be\label{asintoto}
v(x)=\frac{e^{\mathrm{i} k r}}{\sqrt{r}}\left\{v_{\infty}(\hat{x})+O\left(\frac{1}{r}\right)\right\},\ \ \ \  r=|x|\rightarrow\infty,
\ee
where $v_\infty$ is the far-field pattern of $v$. If $v^1$ and $v^2$ are two such solutions, from (\ref{asintoto}) we find \cite{cokr98}
\be\label{zioperplesso}
v^1(x)\frac{\partial v^2(x)}{\partial r}=\mathrm{i} k \frac{e^{2\mathrm{i} k r}}{r}\,v^1_{\infty}(\hat{x})\, v^2_{\infty}(\hat{x}) 
+O\left(\frac{1}{r^2}\right),\ \ \ \ r\rightarrow\infty,
\ee
uniformly for all directions. By applying (\ref{zioperplesso}) to the radiating fields $u^s(\cdot;x_0)$ and $G(\cdot;x)$, we easily 
find that the integrand function at the right-hand side of (\ref{monti}) is $O\left(\frac{1}{r^2}\right)$ and then the integral itself 
vanishes as $r\rightarrow\infty$. This is even more true if the wave number $k>0$ is replaced by $k\sqrt{\tilde{n}_0}$ with 
$\mathrm{Im}\left\lbrace\sqrt{\tilde{n}_0}\right\rbrace >0$, since the attenuation of the fields and their derivatives is faster.

As regards the integral on $B_R$ in (\ref{quiz}), 
we come back to eq. (\ref{errani}) in Section \ref{ziodifferenziale} and recall that the singularity of $u(x;x_0)$ for $x=x_0$ is
cancelled out by $m(x)$, since $x_0\notin\mathrm{supp}\, m$: then, substituting (\ref{errani}) into (\ref{quiz}) and taking the limit 
as $R\rightarrow\infty$, we can write
% we observe that a simple algebraic manipulation of eq. (\ref{scatteringnear})(a) yields
% \be\label{trucco}
% \Delta_x u(x;x_0) +k^2 n_b(x)\, u(x;x_0)=-\delta(x-x_0)+k^2\tilde{n}_0\frac{n_b(x)-n(x)}{\tilde{n}_0}u(x;x_0).
% \ee
% If we now remember eq.s (\ref{graun}) and (\ref{scatteringnear})(b), as well as the identification $u^i(\cdot;x_0)=G(\cdot;x_0)$ and 
% the definitions $\tilde{k}_0:=k\sqrt{\tilde{n}_0}$, $m(x):=[n_b(x)-n(x)]/\tilde{n}_0$, from (\ref{trucco}) we have
% \be\label{barbatrucco}
% \Delta_x u^s(x;x_0) +k^2 n_b(x)\, u^s(x;x_0)=\tilde{k}_0^2 m(x) u(x;x_0).
% \ee
% Note that at the right-hand side of (\ref{barbatrucco}) the singularity of $u(x;x_0)$ for $x=x_0$  is cancelled out by $m(x)$,
% since $\mathrm{supp}\,m=\bar{D}$ is compact and $x_0\notin \mathrm{supp}\,m$ by hypothesis. Then, by substituting (\ref{barbatrucco})
% into (\ref{quiz}) and taking the limit as $R\rightarrow\infty$, we can write
\be\label{lippusrec}
u^s(x;x_0)=-\tilde{k}_0^2\int_{\R^2}G(y;x)m(y)u(y;x_0)\,dy\ \ \ \ \forall x\in\tilde{\Omega}_0,
\ee
which is immediately written in the form (\ref{lippus}) by using the reciprocity property \cite{como98,landau8} $G(y;x)=G(x;y)$.
\end{proof}

We call eq. (\ref{lippus}) the `inhomogeneous Lippmann-Schwinger equation', to emphasize that the reference background is (or may be) 
inhomogeneous. In the next section, we shall apply an inversion algorithm to this equation in order to compute the point values of $m$, 
i.e., of the refractive index inside the region under investigation.

\section{A hybrid scheme}\label{zioibrido}

In general, iterative methods for the quantitative solution of inverse scattering problems are applied to the homogeneous Lippmann-Schwinger equation. Such methods take as input the scattering data collected by antennas placed outside the investigated area and inside a homogeneous zone, and provide as output the reconstruction of the refractive index everywhere in the inhomogeneous region. The main drawback of this computational approach is that, particularly when the scattering experiment is performed with a single and fixed frequency, the number of unknowns is typically much larger than the number of measured data, and therefore the reconstruction accuracy is often rather low.
%(the accuracy could be improved by applying the procedure to multi-frequency data).
However, there are situations where just a certain part of the inhomogeneous region is unknown and of interest for practical applications,
while information is available about the refractive index of the rest of the domain. In this case, the quantitative inverse scattering method can
be applied to the inhomogeneous Lippmann-Schwinger equation, provided one is able to compute the Green's function of the inhomogeneous (and known)
background and to approximately identify the region taken up by the scatterer under investigation.
%qualitatively distinguish between the target scatterer and the background.
In the following of the current section, we describe an implementation of this approach essentially based on the contrast source inversion method.
Therefore, in order to realize the proposed hybrid approach, the ingredients we need are: 1) a method for computing the Green's function of the background; 2) a qualitative method to visualize (an overestimate of) the scatterer support inside the background itself; 3) a quantitative scheme
for the inversion of the inhomogeneous Lippmann-Schwinger equation in the region identified at step 2).

\subsection{The computation of the Green's function}\label{green}

A handy way to compute the Green's function of an inhomogeneous background is given by the method of moments (MOM) \cite{harrington,ri65}. 
Since $G(\cdot;x_0)$ solves a particular case of problem (\ref{scatteringnear}) with the identifications $u(\cdot;x_0)=G(\cdot;x_0)$, 
$u^i(\cdot;x_0)=\Phi(\cdot;x_0)$ and $n(x)=n_b(x)$, it also satisfies the homogeneous Lippmann-Schwinger equation \cite{cokr98,como98}
\be\label{figalli}
G(x;x_0)=\Phi(x;x_0)-\tilde{k}_0^2\int_{\R^2}\Phi(x;y) m_b(y)G(y;x_0)\,dy\ \ \ \ \mathrm{for}\ x\in\R^2,
\ee
with $m_b(x):=[\tilde{n}_0-n_b(x)]/\tilde{n}_0$ and $x_0\in\tilde{\Omega}_0$.
%then it corresponds to the value of the total field measured in $x$ when a cylindrical wave is emitted in $x_0$. Therefore it satisfies
%\begin{equation}\label{lipp-green} G(x,x_0)=-\frac{i}{4}H_{0}^{(1)}(k|x-x_0|)-k^2\int_{\R^2}m(y)G(y,x_0)(\frac{i}{4}H_{0}^{(1)}(k|x-y|))dy
%\end{equation}
%where $H_0^{(1)}$ is the Hankel function of the first kind and of order 0 and where $-\frac{i}{4}H_{0}^{(1)}(k|x-x_0|)$ is the Green's function
%of the background.
Since $\mathrm{supp}\,m_b$ is compact, the integration domain in (\ref{figalli}) is bounded: then, we can consider a finite (and not necessarily uniform) partition of $\mathrm{supp}\,m_b$ by $L$ cells $A_1,\ldots,A_L$, chosen so small that $m_b$ and $G(\cdot;x_0)$ can be assumed
to be constant inside each cell. Then, eq. (\ref{figalli}) can be approximated as
\be\label{alessio}
G(x;x_0)\approx\Phi(x;x_0)-\tilde{k}_0^2\sum_{j=1}^{L} m_b(y_j) G(y_j;x_0) \int_{A_j}\Phi(x;y)dy,
\ee
where $y_j\in\R^2$ identifies the center of the cell $A_j$.
%Assuming as well that $x_0$ and $x$ are outside and inside the support of $m$ respectively (but, by reciprocity, the two positions
%can be switched), then, equation (\ref{lipp-green}) can be approximated as
%\begin{eqnarray}\label{lipp-green-approx}
%G(x,x_0) \simeq \frac{i}{4}H_{0}^{(1)}(k|x-x_0|)- \nonumber \\
%k^2\sum_{j=1}^{L} m(x_j) G(x_j,x_0) \int_{cell_j}\frac{i}{4}H_{0}^{(1)}(k|x-y|)dy.
%\end{eqnarray}
%where $x_j$ identifies the center of each cell.
Furthermore, following \cite{ri65}, we approximate each square cell as a circular cell of the same area, so that the integral in eq.
(\ref{alessio}) can be evaluated as
\begin{align}\label{integral-approx}
\int_{A_j}\Phi(x;y)dy & =\frac{\I}{4}\int_{A_j}H^{(1)}_0(\tilde{k}_0|x-y|)dy\approx \\   \nonumber
& \approx
\left\{
\begin{array}{ll}
\displaystyle\frac{\I\pi \tilde{k}_0 a_j}{2}J_{1}(\tilde{k}_0 a_j)H_{0}^{(1)}(\tilde{k}_0|x-y_j|)  & \mbox{for } x \notin A_j\\ [2mm]
\displaystyle\frac{\I}{2}\left[\pi \tilde{k}_0 a_j H_{1}^{(1)}(\tilde{k}_0 a_j)+2\I\right]  & \mbox{for } x \in A_j,
\end{array}
\right.
\end{align}
where $a_j:=\sqrt{\Delta_j/\pi}$, $\Delta_j$ is the area of the cell $A_j$, $J_1$ is the Bessel function of first order and $H_{1}^{(1)}$ is the Hankel function of first kind of order one. We point out that the method can be notably speeded up by utilizing the fast Fourier transform algorithm as it is described for example in \cite{bozza09,zhetal03}: we shall adopt this version of the MOM method for our numerical simulations.

\subsection{The linear sampling method}\label{basta}

The linear sampling method (LSM) is the earliest and most used qualitative method in inverse scattering: in the case of an object embedded in a homogeneous background, it provides a reconstruction of its support by only knowing the field measured around it.
When the background is inhomogeneous and its Green's function is known, the LSM
%for the qualitative visualization of a scatterer immersed in a homogeneous background,
can be extended to the case of an inhomogeneity immersed in a medium with piecewise constant refractive index.
The basic idea is to write and to approximately solve the modified far-field equation \cite{cafaha06,como98}
\begin{equation}\label{mod-far-field}
\int_{\Gamma} [u^s(x;x_0)-u^s_b(x;x_0)] g_z(x_0)\, dx_0 = G(x;z)\ \ \ \ \mathrm{for}\ x \in \Gamma,
\end{equation}
where $\Gamma:=\{x\in\R^2 : |x|=R_{\Gamma}\}\subset\Omega_0$ is the circle of radius $R_{\Gamma}$ where emitting and receiving antennas are placed,
and $z\in\R^2$ is a sampling point inside the investigation domain enclosed by $\Gamma$. In (\ref{mod-far-field}), the Green's function $G(x;z)$
and the field $u^s_b(x;x_0)$ can be computed by exploiting the knowledge of the background and applying the MOM method, while $u^s(x;x_0)$ represents
the measurements, i.e., the data of the inverse scattering problem.

In fact, it can be proved \cite{cafaha06,como98} that there exists an approximate solution of (\ref{mod-far-field}) whose $L^2(\Gamma)$-norm is bounded inside the scatterer, grows up as $z$ approaches its boundary from inside and remains very large outside, thus behaving as an indicator function for the support $\bar{D}$ of the scatterer itself.
This result inspires a simple algorithm that approximately solves the modified far-field equation by means of a regularization procedure. More precisely, the LSM requires the choice of a numerical grid covering the region where the scatterer is placed;
then, for each point $z$ of the grid, a discretized version of the modified far-field equation is solved by using the Tikhonov regularization
method \cite{tietal95}, and the Euclidean norm of the regularized solution is plotted for each $z$. As a result, the boundary of the scatterer is highlighted by the points of the grid corresponding to the largest increase of the Euclidean norm. We recall that the computational time of this
algorithm can be notably reduced by applying a no-sampling formulation \cite{arbrpi06}, called `no-sampling linear sampling method' (NSLSM), which is the one adopted in the present paper.

Finally, we remark that, for the purposes of our implementation, a univocal identification of the shape of the scatterer is performed by post-processing the NSLSM through an active contour technique \cite{arbrcopi08,chve01}: for reasons that will be clear in the next subsection, we choose the parameters of this edge detection algorithm in such a way that a slight overestimate of the scatterer is favoured.

\subsection{The contrast source inversion method}\label{CSI}

In our hybrid scheme, the contrast source inversion (CSI) method \cite{giletal09,vdbkl97,vdbbrab99} is the technique we use for the quantitative
inversion of the inhomogeneous Lippmann-Schwinger equation. More precisely, the CSI method computes the point values of the refractive index $n(x)$ inside an investigation domain $T$ containing the scatterer support $\bar{D}=\mathrm{supp}\,m$.
The idea at the basis of the CSI method is to split an inverse scattering problem described by an integral equation such as (\ref{lippus}) into two different problems, one defined inside $T$ and the other one on the curve $\Gamma$ where antennas are placed.
More formally, remembering (\ref{scatteringnear})(b), equation (\ref{lippus}) is equivalently recast in the form
\be\label{ferrer}
u(x;x_0)=u^i(x;x_0)-\tilde{k}_0^2 \int_{T}G(x;y)m(y)u(y;x_0)\,dy,
\ee
and then replaced by the system
\begin{equation}\label{lipp-sch-system-comp}
\left\{
\begin{array}{lll}
\displaystyle w(x;x_0)={m}(x)u^i(x;x_0)-m(x)\left[{\mathcal{G}}^{T}w(\cdot;x_0)\right](x) & \mbox{for $x\in T$} & \mathrm{(a)} \\
\displaystyle u^{s}(x;x_0)=-\left[{\mathcal{G}}^{\Gamma}w(\cdot;x_0)\right](x) & \mbox{for $x\in \Gamma$}, & \mathrm{(b)}
\end{array}
\right.
\end{equation}
where $w(x;x_0):=m(x)u(x;x_0)$, the operator ${\mathcal{G}}^{T}:L^2(T)\rightarrow L^2(T)$ is defined as
\begin{equation}\label{operators1}
\left[{\mathcal{G}}^{T} f\right](x):=\tilde{k}_0^2\int_{T} G(x;z)f(z)\,dz \qquad \forall x\in T,
\end{equation}
and ${\mathcal{G}}^{\Gamma}:L^2(T)\rightarrow L^2(\Gamma)$ is defined as
\begin{equation}\label{operators2}
\left[{\mathcal{G}}^{\Gamma} f\right](x):=\tilde{k}_0^2\int_{T} G(x;z)f(z)\,dz \qquad \forall x \in \Gamma.
\end{equation}
Note that eq. (\ref{lipp-sch-system-comp})(a) is obtained from eq. (\ref{ferrer}) by multiplying both members for $m(x)$, while eq. (\ref{lipp-sch-system-comp})(b) is the restriction of eq. (\ref{lippus}) for $x\in \Gamma$: in particular, $u^s(x;x_0)$ for $x\in\Gamma$ represents
the data of the inverse scattering problem.

The CSI method consists of minimizing the functional
\begin{equation}
F(w,m)= \frac{\|u^{s}-{\mathcal{G}}^{\Gamma}w\|^2_{\Gamma}}{\|u^{s}\|^2_{\Gamma}} +
\frac{\|{m u^i - w - m\,}{\mathcal{G}}^{T} w \|^2_{T}}{\|mu^i\|^2_{T}},
\end{equation}
where we dropped the dependence on $x$, $x_0$, and $\|\cdot\|_{\Gamma}$, $\|\cdot\|_{T}$ denote the $L^2$-norms on the spaces $L^2(\Gamma)$, $L^2(T)$ respectively. The minimization of $F$ is performed by applying gradient algorithms \cite{polak} and the updates are alternately computed for $m$ and $w$ \cite{vdbkl97,vdbbrab99}. The information on the background is again coded into the Green's function, i.e., into the operators $\mathcal{G}^{T}$
and $\mathcal{G}^{\Gamma}$.

In the standard case of a homogeneous medium, the implementation of the approach described above is well-established and does not require further discussion. Instead, in the inhomogeneous case two issues need to be addressed. First, we observe that eq. (\ref{lipp-sch-system-comp})(a) is written for $x\in T$; on the other hand, the integral equation (\ref{lippus}), whence eq. (\ref{lipp-sch-system-comp})(a) should derive, has been proved in Theorem \ref{teorema} only for $x\in \tilde{\Omega}_0$ and, in general, $T$ is not contained in $\tilde{\Omega}_0$.
To overcome this drawback, we observe that the purpose of the CSI algorithm is to compute the values of $n(x)$ for $x\in T$, and these values
do not depend on the reference background inside $T$. Then,
%However,
%the background is physically well-defined only outside the region $D$ taken up by the scatterer, while it is somewhat arbitrary inside: then,
we can think that the original background in $T$ is replaced by an artificial one, which hosts the same homogeneous medium occupying $\tilde{\Omega}_0$.
In other words, this amounts to replacing the refractive index $n_b$ with $\tilde{n}_b$, such that $\tilde{n}_b(x)=\tilde{n}_0$ for all $x\in T$ and
$\tilde{n}_b(x)=n_b(x)$ for all $x\notin T$. Of course, the background Green's function $G(x;z)$ in eqs. (\ref{operators1}) and (\ref{operators2})
must be replaced accordingly, i.e., by $\tilde{G}(x;z)$.
%: the computation of the latter can be performed just as outlined in subsection \ref{green}.
% Eventualmente precisare che $n(x)$ deve essere diverso da $n_0$, altrimenti non vedo più lo scatterer.
We point out that a good choice of the investigation domain $T$ is the region highlighted by the NSLSM, as explained in Subsection \ref{basta}: indeed, such region is an overestimate of $D$, i.e., it contains $D$, but is also close to be as small as possible,
thus minimizing the number of points where $n(x)$ is to be computed, i.e., the number of unknowns of the problem. As a result, also the number of measurements necessary for a successful implementation of the method is optimized.

This same trick, i.e., changing the background as just explained, is also useful to address the second issue, which is concerned with the
computation of the Green's function $G(x;z)$. Indeed, eq. (\ref{operators1}) shows that both $x$ and $z$ vary in $T$: then, in particular, also the singular case $x=z$ is of interest. Now, $G(x;z)$ satisfies eq. (\ref{figalli}) (with the identification $z=x_0$), which implies that if $z\in\mathrm{supp}\,m_b$ and $x=z$, the integrand function at the right-hand side of (\ref{figalli}) is affected by a double singularity for $y=z$:  one due to $\Phi(z;y)$, the other one due to $G(y;z)$. This prevents us from applying the usual approximation scheme outlined in eq.s (\ref{alessio}) and (\ref{integral-approx}): accordingly, an \textit{ad hoc} quadrature rule should be created to numerically compute the integral. Such problem is avoided by the previous choice of the new background, i.e., of $\tilde{n}_b$: indeed, the region $T$ is now erased by the actual integration domain in (\ref{figalli}), since $\tilde{m}_b(x):=[\tilde{n}_0-\tilde{n}_b(x)]/\tilde{n}_0$ vanishes inside $T$, while $z\in T$.

%The theoretical framework outlined above
We are now ready to present in the next section some numerical results obtained by implementing the theoretical framework developed so far.

\section{Numerical examples} \label{zionumerico}
The aim of this section is twofold. First, we numerically validate the CSI algorithm in the inhomogeneous case: more precisely, we compare the performance of the CSI based on the inhomogeneous Lippmann-Schwinger equation with that of the traditional CSI (the two algorithms will be called `inhomogeneous CSI' and `homogeneous CSI', respectively). To this end, we perform a preliminary set of simulations by using simple numerical phantoms and without adding noise to measurements. Second, we test the effectiveness of the whole hybrid scheme presented above, as well as its worthwhileness in real applications, by implementing it with a realistic phantom of a female breast slice: the goal is to highlight the presence of a tumoral mass inside the inhomogeneous background formed by the healthy tissues.

%This section is dedicated to a numerical validation of the hybrid inversion technique presented above. A first group of experiments is
%dedicated to assess the effectiveness of the inhomogeneous Lippmann-Schwinger approach with respect to the traditional homogeneous-one:
%in this sense, a comparison between the operations of both (homogeneous and inhomogeneous-background) CSI techniques, when applied to
%simple numerical phantoms, will be presented. Then, in a second group of numerical experiments, the effectiveness of the whole hybrid method,
%as well as its worthwhileness for real applications, will be shown through some realistic simulations concerning breast cancer detection.

\subsection{Homogeneous and inhomogeneous CSI: a comparison}

\begin{figure}%[H]
\begin{center}
\begin{tabular}{cc}
\epsfig{file=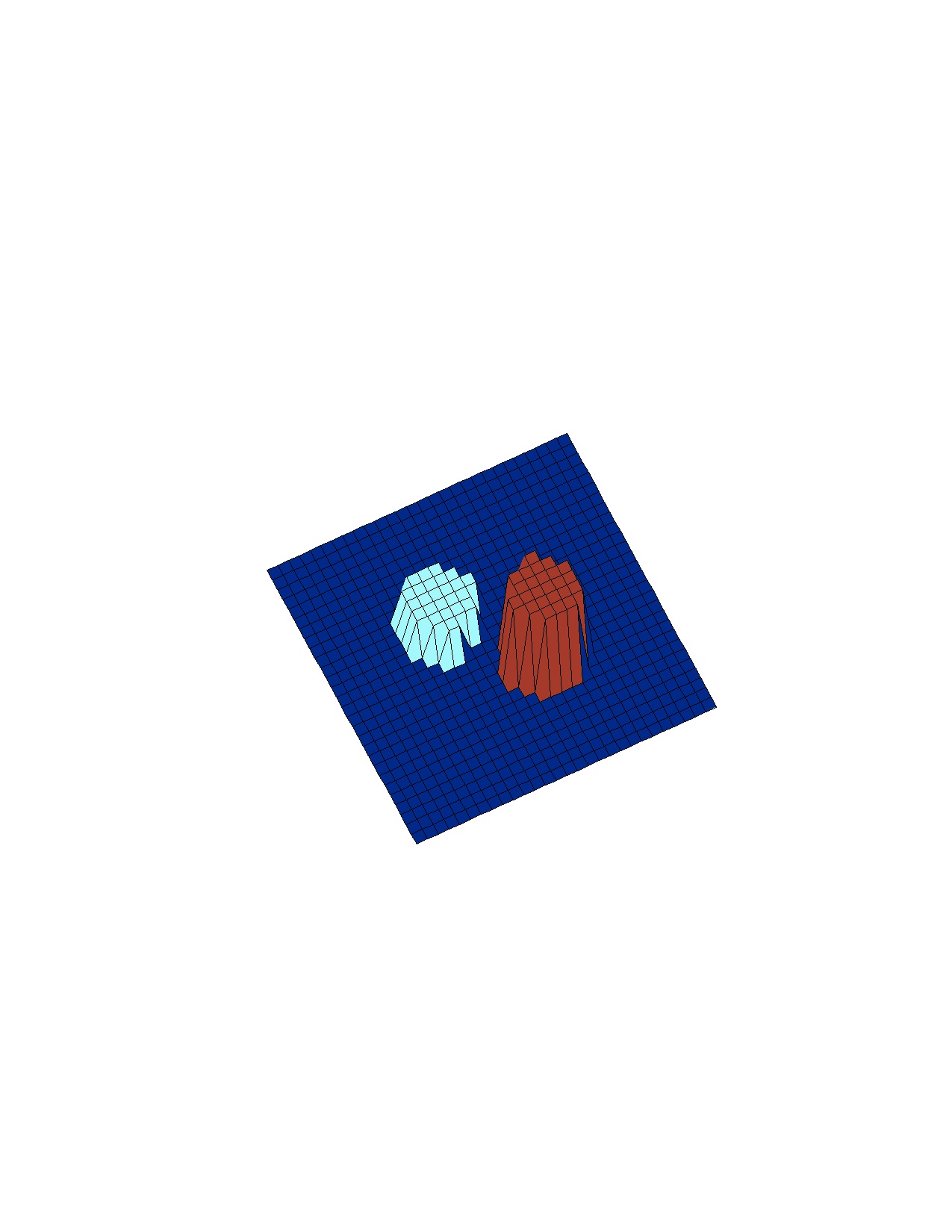, width=6.0cm}&
\epsfig{file=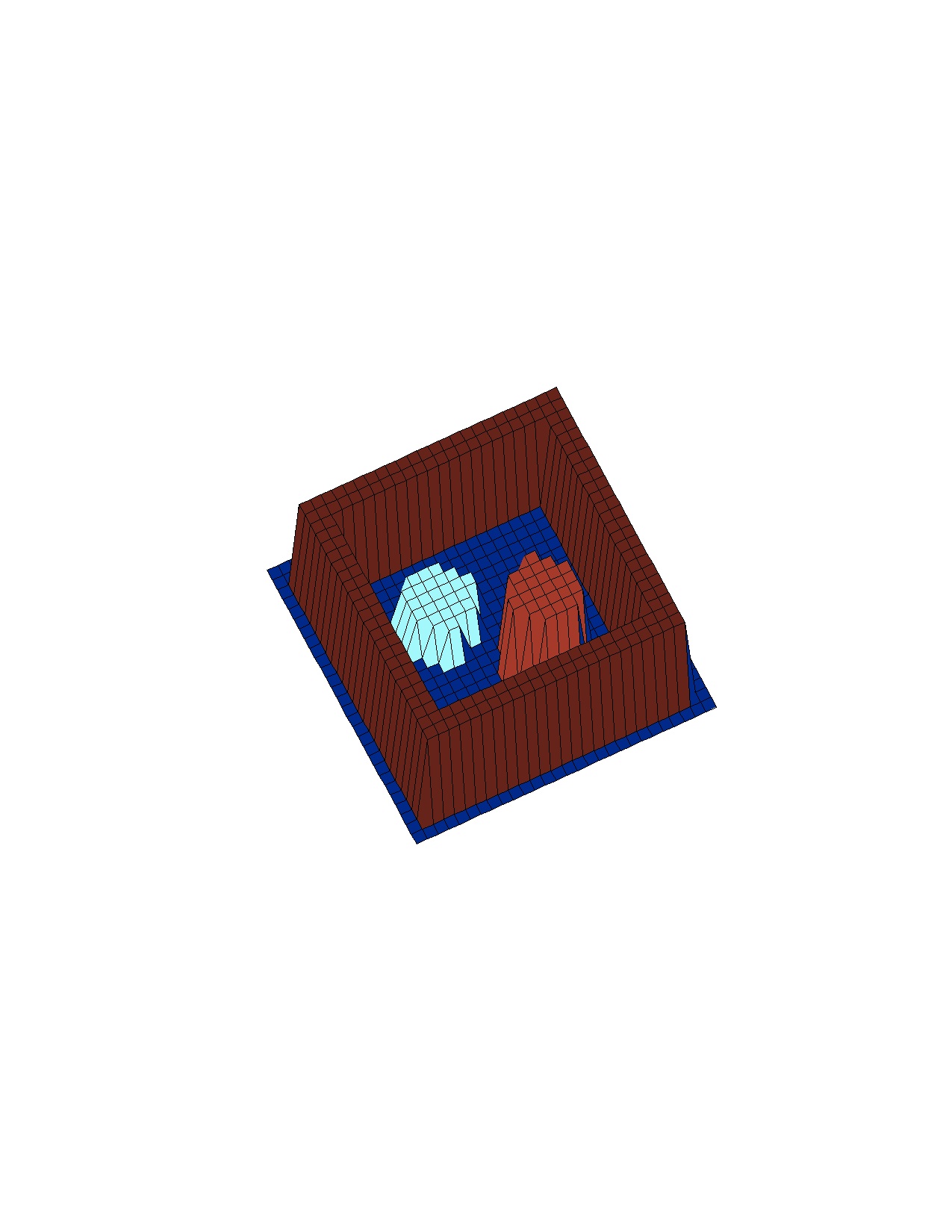, width=6.0cm}\\
(a) & (b)
\end{tabular}
\caption{\label{fantocci}Visualization of the relative dielectric permittivity $\varepsilon_r$ of the numerical phantoms utilized for the first set of scattering experiments.
The phantoms are piecewise homogeneous and purely dielectric, i.e., non-absorbing.}
%Value of the contrast functions showed in Figure \ref{fantocci_reconstruction} (a) and (c).}
\end{center}
\end{figure}

The first set of numerical simulations is concerned with the reconstruction of the two (non-absorbing) phantoms visualized in panels (a) and (b) of Figure \ref{fantocci}. In particular, we are going to show an example where the homogeneous CSI fails to provide a satisfactory result, while the inhomogeneous CSI succeeds: this is obtained by coding some additional information into the background, which then becomes inhomogeneous.

Indeed, consider first the phantom in Figure \ref{fantocci}(a): the pixel values of the relative dielectric permittivity $\varepsilon_r$ and the geometry of the objects are plotted in Figure \ref{fantocci_reconstruction}(a). The scattering experiment is simulated by choosing a wavelength
$\lambda$ equal to the length unit adopted for all panels in Figure \ref{fantocci_reconstruction},
and by using $30$ unit point sources uniformly placed on a circle centred at the origin and with radius $3\,\lambda$. The scattered field is computed
by means of a MOM code, as outlined in Subsection \ref{green}, at $30$ points obtained from the previous ones after a rotation of $\pi/30$.
Then, we apply the homogeneous CSI (implemented in the error-reducing version of \cite{vdbbrab99} and initialized by backpropagation): as a result, we obtain the satisfactory reconstruction shown in Figure \ref{fantocci_reconstruction}(b).

However, if we consider the more complex scenario of Figure \ref{fantocci}(b), with pixel values of $\varepsilon_r$ plotted in Figure
\ref{fantocci_reconstruction}(c), the same homogeneous CSI provides the unsatisfactory reconstruction shown in Figure \ref{fantocci_reconstruction}(d): it is evident that the algorithm converges to the wrong local minimum. A feasible way to overcome this drawback is to consider the square barrier surrounding the two disks as a part of the background: this situation is represented in Figure \ref{fantocci_reconstruction}(e). Then, we implement the inhomogeneous CSI: in particular, the reconstruction is now performed only inside the barrier, i.e., inside the green square highlighted in Figure \ref{fantocci_reconstruction}(f). As a result, the effectiveness of the algorithm in reconstructing the two disks is restored, as shown by Figure \ref{fantocci_reconstruction}(f) itself.

\begin{figure}%[H]
\begin{center}
\begin{tabular}{cc}
\psfig{figure=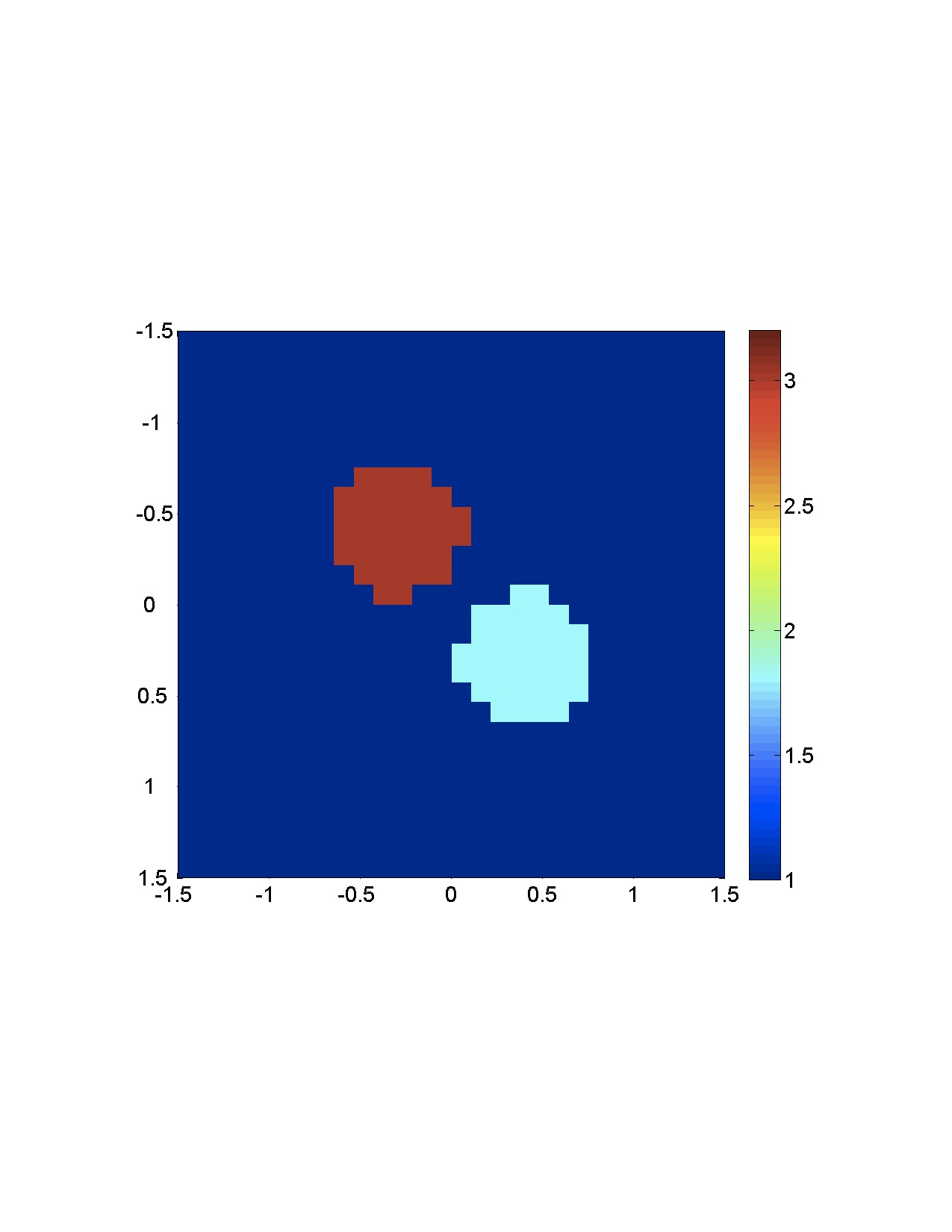, width=6.0cm} &
\psfig{figure=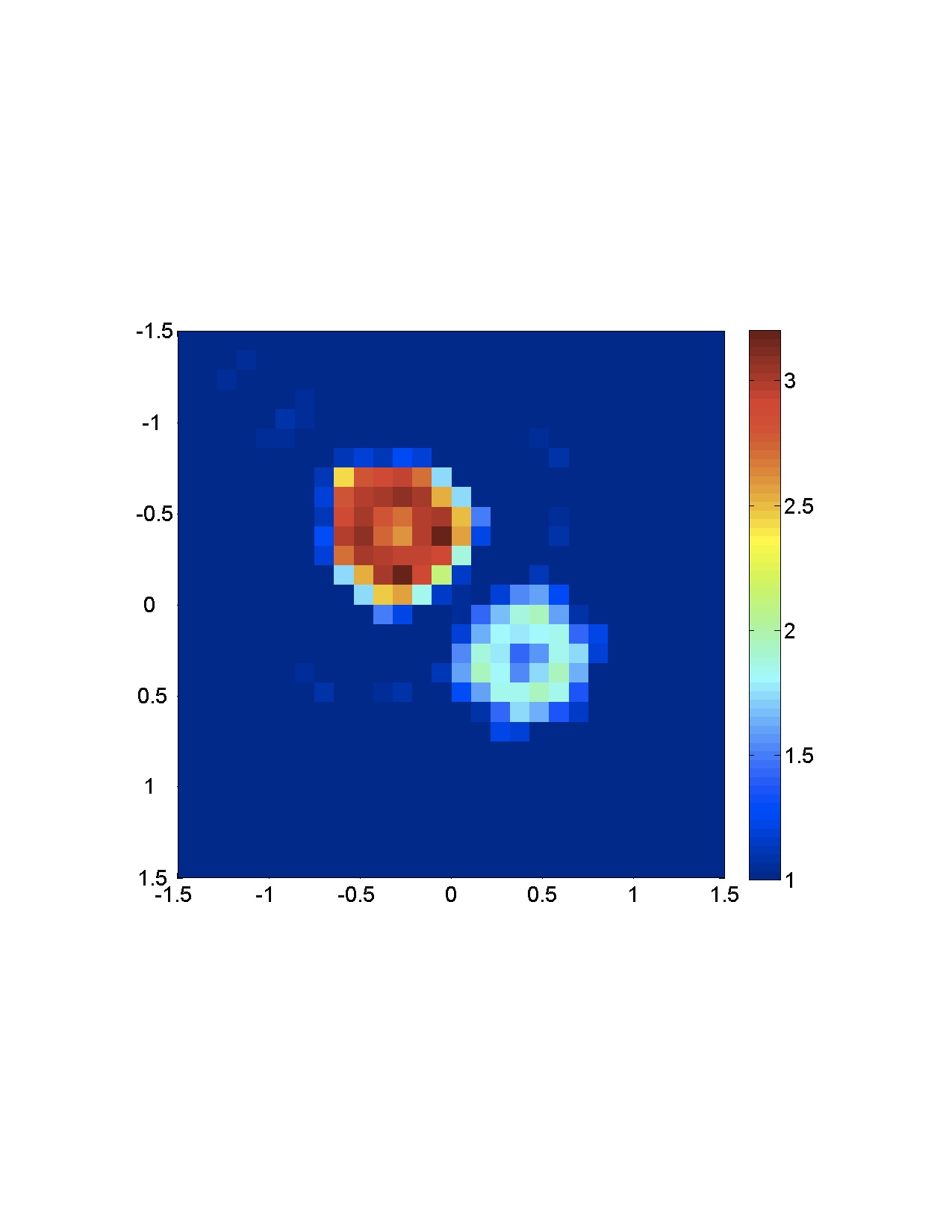, width=6.0cm}\\
$\ \ \ \ \ $ (a) & $\ \ \ \ \ $ (b) \\[0.35cm]
\psfig{figure=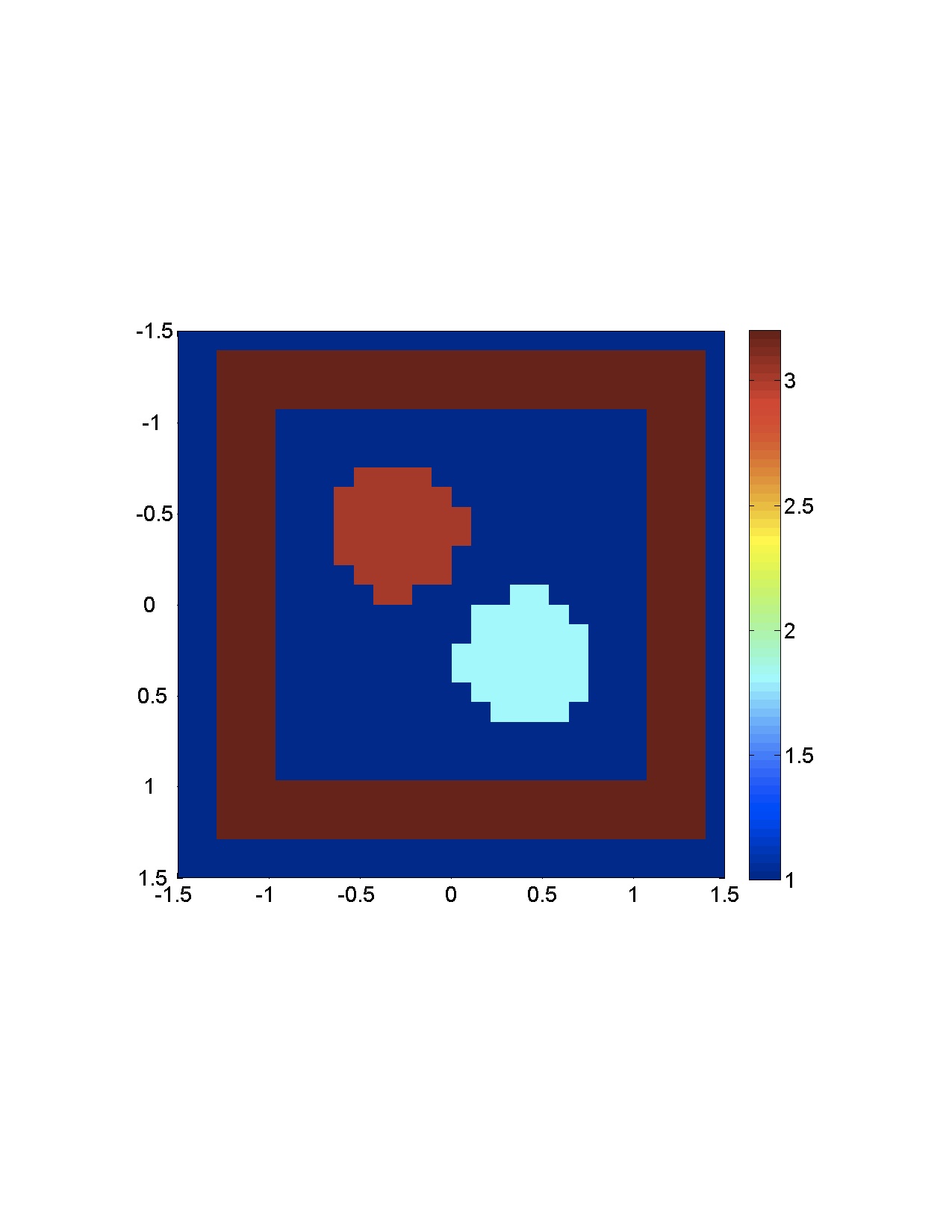, width=6.0cm} &
\psfig{figure=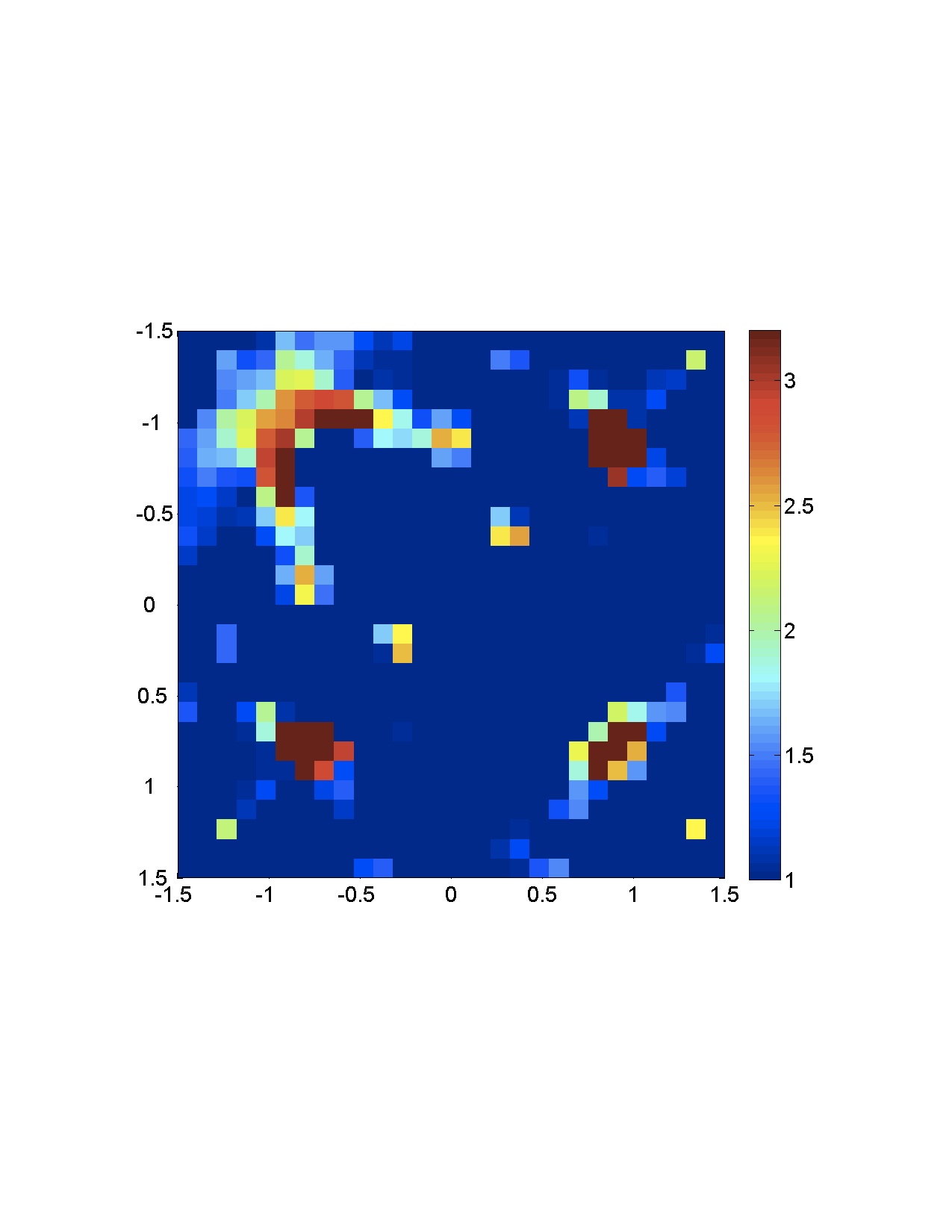, width=6.0cm} \\
$\ \ \ \ \ $ (c) & $\ \ \ \ \ $ (d) \\[0.35cm]
\psfig{figure=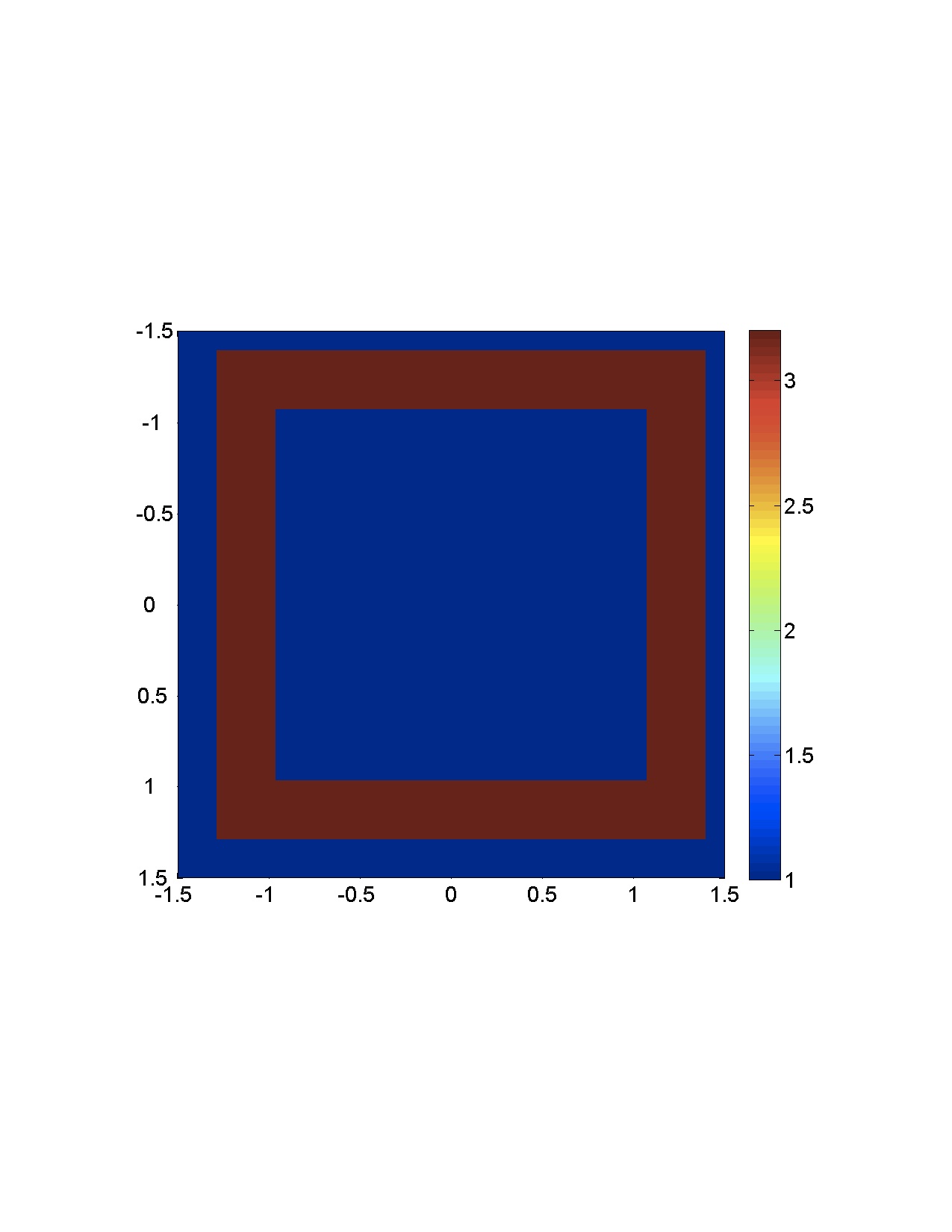, width=6.0cm}&
\psfig{figure=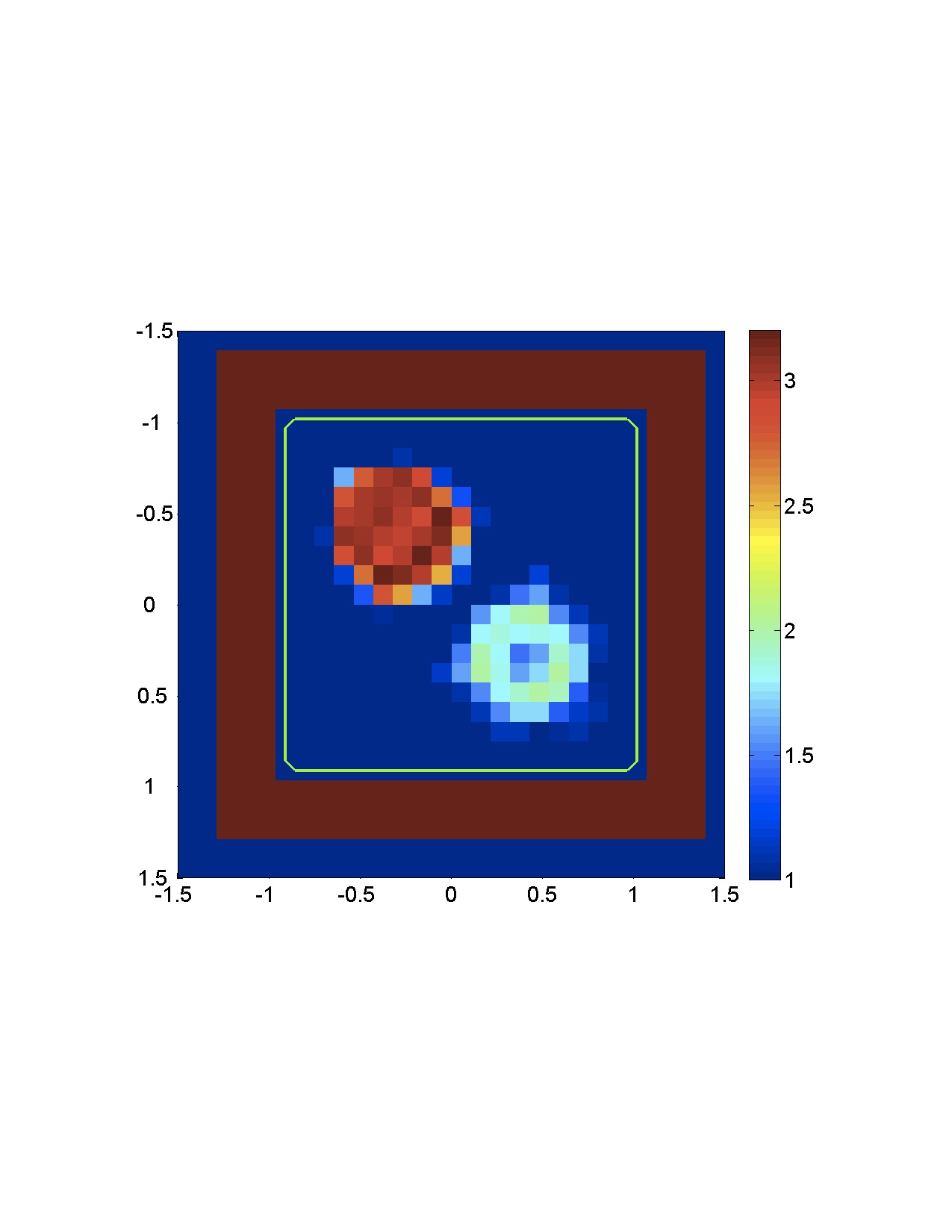, width=6.0cm} \\
$\ \ \ \ \ $ (e) & $\ \ \ \ \ $ (f) \\[0.35cm]
\end{tabular}
\caption{\label{fantocci_reconstruction}Panels (a) and (c): pixel values of the relative dielectric permittivity $\varepsilon_r$ of the phantoms shown in Figure (a) and Figure (b), respectively. Panels (b) and (d): corresponding reconstructions provided by the homogeneous CSI. Panel (f): reconstruction of the phantom in (c) provided by the inhomogeneous CSI, when the background is assumed to be as shown in panel (e).}
\end{center}
\end{figure}

\subsection{The hybrid approach validated against a realistic numerical breast phantom}

In the following numerical examples we test the effectiveness of the hybrid scheme described in Section \ref{zioibrido} when applied to scattering situations that might occur in practice.
To this end, we consider a realistic numerical phantom of a female breast slice, whose pixel values of relative permittivity $\varepsilon_r$ and conductivity $\sigma$ are obtained from the segmentation of MRI (Magnetic Resonance Imaging) images \cite{laetal07} (such phantoms are downloadable from the web site \texttt{uwcem.ece.wisc.edu/home.htm}). Then, a circular tumor with a diameter of $0.5\,$cm
is artificially inserted into the phantom: we refer e.g. to \cite{laoketal07} for the values of the electric parameters of tumoral tissues at various frequencies.
Our numerical experiments are performed at a frequency of $f=3\,$GHz: accordingly, in panels (a) and (b) of Figure \ref{breast_exact} we plot the pixel values of the relative permittivity $\varepsilon_r$ and conductivity $\sigma$ of the breast tissues at such frequency. As shown in this figure, we also assume that the breast slice is immersed in an infinite and homogeneous coupling medium \cite{meetal03}:
this trick reduces reflection phenomena, thus favouring the penetration of the wave into the breast itself. Finally, we choose $30$ unit point sources uniformly placed on a circle surrounding the breast, concentric with it and having a radius of $6\,$cm. By using a MOM code,
the scattered field is computed at $30$ points obtained from the previous ones after an angular shift of $\pi/30$,
then these field values are perturbed by $3\%$ Gaussian noise.

%With the following numerical examples we want to test the efficacy of the hybrid technique when applied to realistic scattering situations.
%The set is a $2D$ mammographic session where the index of refraction of the scatterer (i.e. breast) is provided by highly realistic numerical
%phantoms\footnote{downloadable at uwcem.ece.wisc.edu/home.htm} obtained from the segmentation of MRI (Magnetic Resonance Imaging) images
%\cite{laetal07} to whom a circular tumor of diameter $0.5$ cm has been artificially added (we refer to \cite{laoketal07} for the computation
%of the value of index for tumoral tissues at fixed frequencies). All the experiments are performed at frequency $f=3 GHz$, measuring the field
%on $30$ antennas placed on a circle of ray $6$ cm surrounding the breast. The incidence directions are $30$ and the field values are perturbed
%with a $3\%$ of noise.

%In panels (a) and (b) of Figure \ref{breast_exact} we plot the pixel-values of the relative permittivity $\epsilon_r (x)$ and conductivity
%$\sigma (x)$ of the breast tissues at a frequency of $f=3 GHz$.
%We remind that the index of refraction is given by $n(x)=\frac{1}{\epsilon_{0}}[\epsilon(x)+\I\frac{\sigma(x)}{\omega}]$, where
%$\epsilon_0$ is the electric permittivity of the hosting medium. Although, for clarity, in all the figures presented in this subsection
%the refractive index is plotted considering $\epsilon_0$ as the permittivity of the vacuum, the breast is surrounded by a matching layer
%\cite{meetal03} of refractive index equal to $18+\I 0.06$ which embeds also the curve where antennas are placed.

A whole and precise reconstruction of such a heterogeneous scatterer as a breast slice is extremely difficult to be achieved: the main issue is
the strong variability of the refractive index from point to point, as well as its large range of values. Then, in order to properly describe the heterogeneity of the tissues, the investigation domain should be discretized by using a large number of pixels, which corresponds to a large number of unknowns in the CSI algorithm. In practical applications, this number can be too large with respect to the amount of data, thus impairing the
reliability of the reconstruction. Therefore, an inhomogeneous-background approach to the problem can be helpful: indeed, when \textit{a priori} known, a part of the scatterer can be included in the background, so that the reconstruction can be performed on a fewer number of pixels (i.e., in a smaller investigation domain), which reduces the number of unknowns and increases the quality of the reconstruction itself.

%A whole and precise reconstruction of a scatterer such heterogeneous as breast is extremely difficult to be achieved: the main difficulty
%lies in the heterogeneity of the refractive index and in the wideness of the range of values that it can assume from point to point and from
%tissue to tissue. Moreover, in order to properly describe the variety of the tissues, a large number the unknowns has to be taken into account:
%in practical applications, their number can be disproportionate to the number of field data and the effective achievability of a truthful
%reconstruction can be out of range. For these reasons, an inhomogeneous-background approach to the problem can be helpful in the sense that,
%since a part of the scatterer can be assumed as background (when a priori known), the reconstruction can be performed on a fewer number of pixels,
%reducing the number of unknowns and increasing the quality of the reconstruction.

In the first numerical example, whose results are collected and visualized in Figure \ref{breast_and_tumor}, we assume that the known background is formed by the fat internal tissue, the skin layer and the coupling medium.
%Since our goal is to reconstruct both the glandular and tumoral tissues inside the breast fat, we have replaced the electric parameters of
%these tissues with those of free space
In general, we may think that such preliminary information on the tissues is obtained from the segmentation of the most recent MRI image of the
patient's breast, while the properties of the coupling medium are known \textit{a priori}.
The pixel values of the relative permittivity $\varepsilon_r$ and conductivity $\sigma$ of these media are shown in panels (a) and (b) of Figure \ref{breast_and_tumor}.
Since our goal here is to reconstruct both the glandular and tumoral tissues inside the breast fat, such background panels have been respectively obtained from panels (a) and (b) of Figure \ref{breast_exact} by removing the artificial circular tumor (i.e., by restoring the original healthy tissue) and by artificially replacing the central glandular mass with an `average fat' (i.e., by replacing
the pixel values of $\varepsilon_r$ and $\sigma$ in the glandular tissue with constant values computed by averaging over the pixel values of
$\varepsilon_r$ and $\sigma$ in the fat tissue).
%we have artificially replaced the electric parameters of such tissues with those of free space (i.e., $\varepsilon_r=1$ and $\sigma=0$):
%this explains the difference between panels (a), (b) of Figure \ref{breast_exact} and panels (a), (b) of Figure \ref{breast_and_tumor}, respectively.
Then, this
background is utilized to implement the NSLSM, the qualitative technique outlined in Subsection \ref{basta}. In Figure \ref{breast_and_tumor}(c) we show the visualization of the unknown region as provided by the NSLSM: more precisely, we plot the values of the indicator function, chosen as the reciprocal of the squared norm of the regularized solution to the modified far-field equation (\ref{mod-far-field}). Next, we apply an active contour technique \cite{arbrcopi08,chve01} to this visualization map: as a result, we extract the support of the unknown region $T$, which is the
homogeneous inner domain shown in panel (d) of Figure \ref{breast_and_tumor}. More precisely, in this panel we plot the pixel values of
$\varepsilon_r$ characterizing the artificial background we need to consider for applying the inhomogeneous CSI in the investigation domain $T$ (the plot of $\sigma$ would be analogous and is not shown here). Indeed, we remind that, in order to avoid the theoretical and computational problems described in Subsection \ref{CSI}, the domain $T$ should be considered as filled with the outmost medium (i.e., in our case, with the coupling medium).
The reconstructions provided by the inhomogeneous CSI inside $T$ for the relative permittivity and for the conductivity are presented in panels (e) and (f) of Figure \ref{breast_and_tumor}, respectively: the results are in rather good agreement with the true phantoms of Figure \ref{breast_exact}.

%the idea is to reconstruct both glandular and tumoral tissues using, as background, the a priori information on skin layer, fat layer and
%matching fluid around the breast that can be provided, for example, by the segmentation of the most recent MRI image of the patient's breast.
%Figure \ref{breast_and_tumor} (a) and (b) show the value of permittivity and conductivity of the background utilized for the application of the
%nSLSM, the qualitative technique which allows the identification of the unknown region. (c) is the reconstruction provided by the nSLSM
%(i.e. plot of one over the squared norm of the solution to the far field equation); from (c), through the application of active-contours
%techniques, the support of the unknown region is extracted. (d) represents value of the permittivity of the background used for applying
%the inhomogeneous-background CSI (the corresponding conductivity is analogous): we remind that, in order to avoid the occurrence of the
%theoretical and computational problems described in Section \ref{CSI}, the value of the hosting medium has to be substituted inside the
%unknown region so that the integral of (\ref{figalli}) vanishes inside the area under investigation. The inhomogeneous-background CSI
%reconstruction is presented in (e) for the relative permittivity and (f) for the conductivity.

\begin{figure}%[H]
\begin{center}
\begin{tabular}{cc}
\epsfig{file=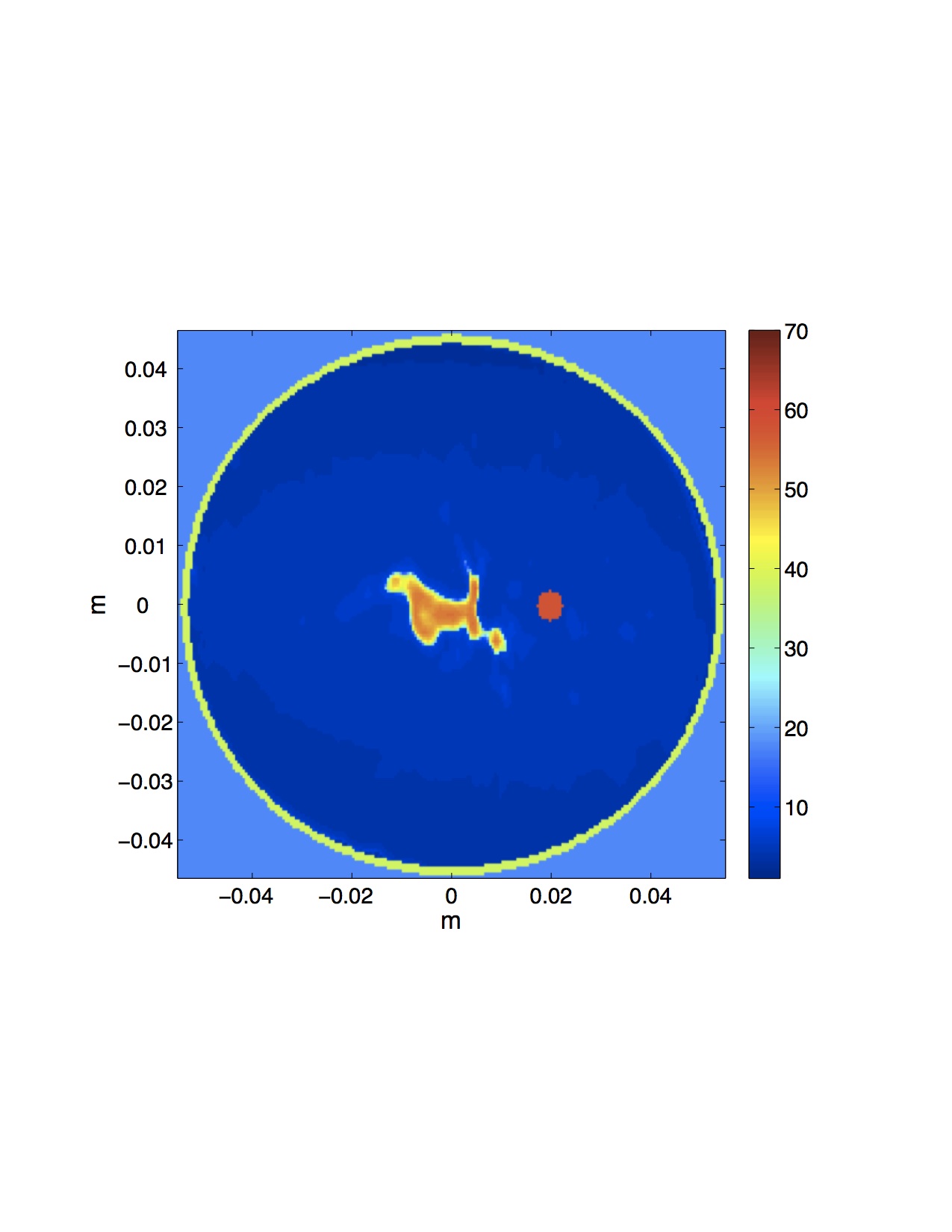, width=6.0cm}&
\epsfig{file=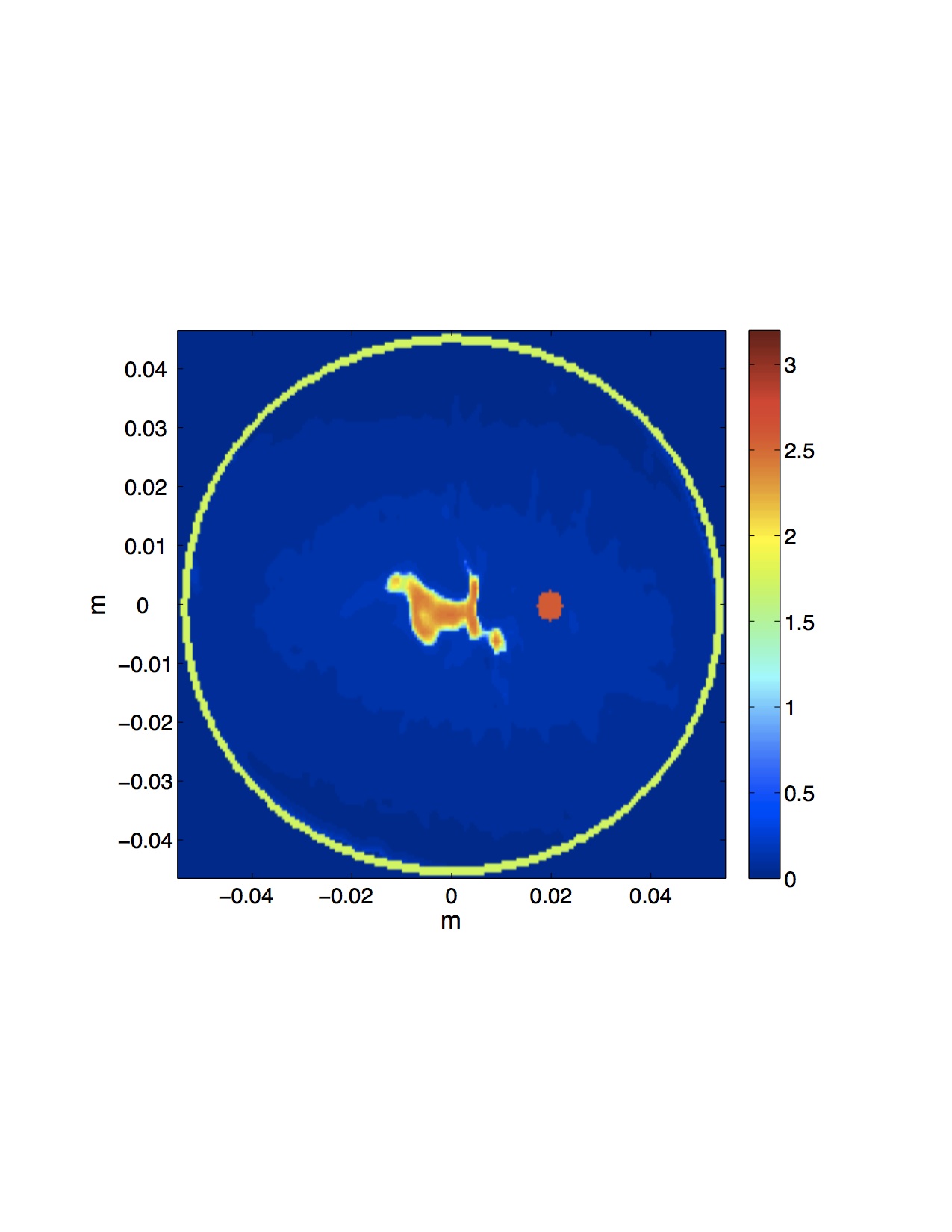, width=6.0cm}\\
(a) & (b)
\end{tabular}
\caption{\label{breast_exact}Numerical phantom of the breast slice utilized in our scattering simulations, performed at a frequency of $f=3\,$GHz. Panel (a): pixel values of the relative dielectric permittivity $\varepsilon_r$. Panel (b): pixel values of the conductivity $\sigma$. In both panels, the tumoral mass is the small disk centered at $(0.02,0)\,$m and with a diameter of $0.5\,$cm.}
\end{center}
\end{figure}

\begin{figure}%[H]
\begin{center}
\begin{tabular}{cc}
\psfig{figure=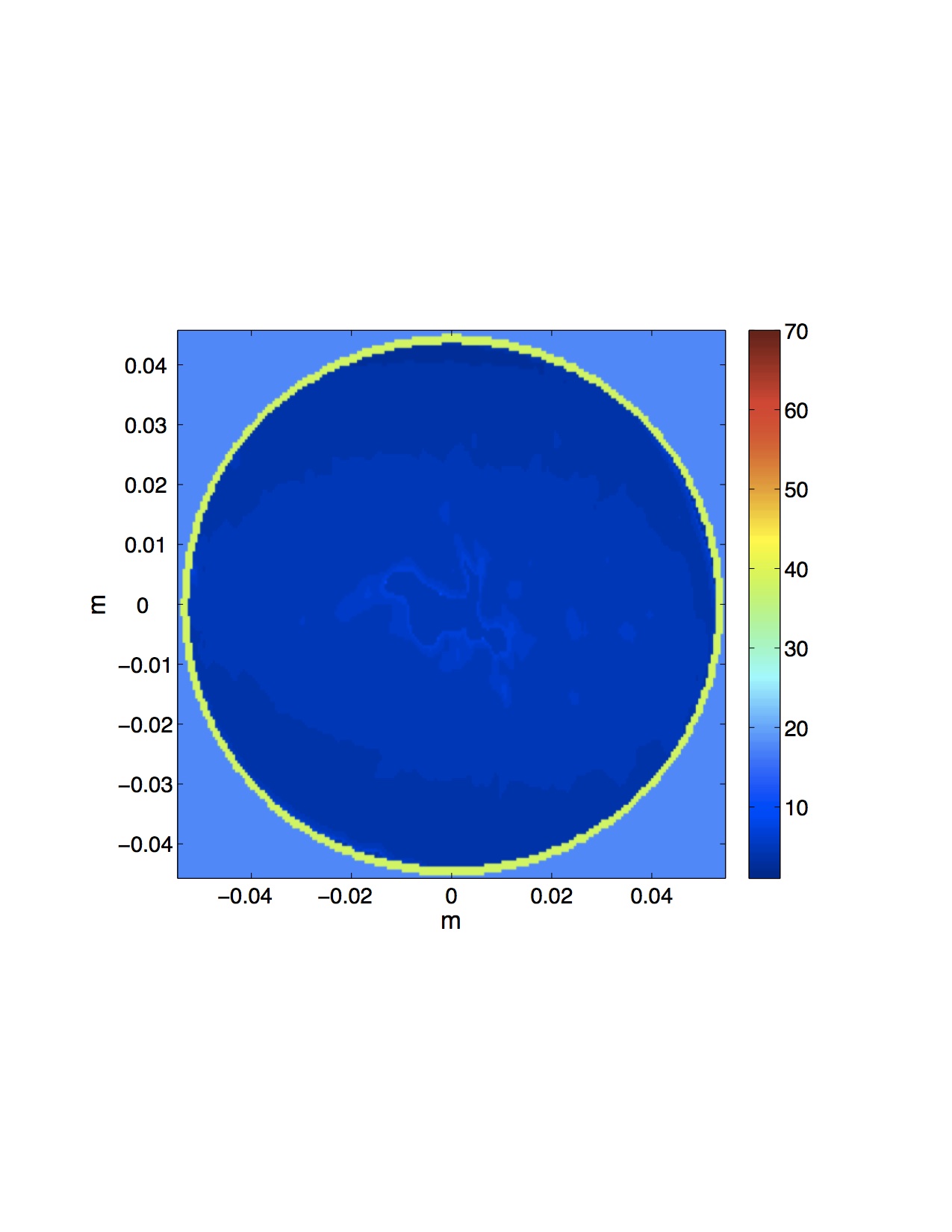, width=6.0cm} &
\psfig{figure=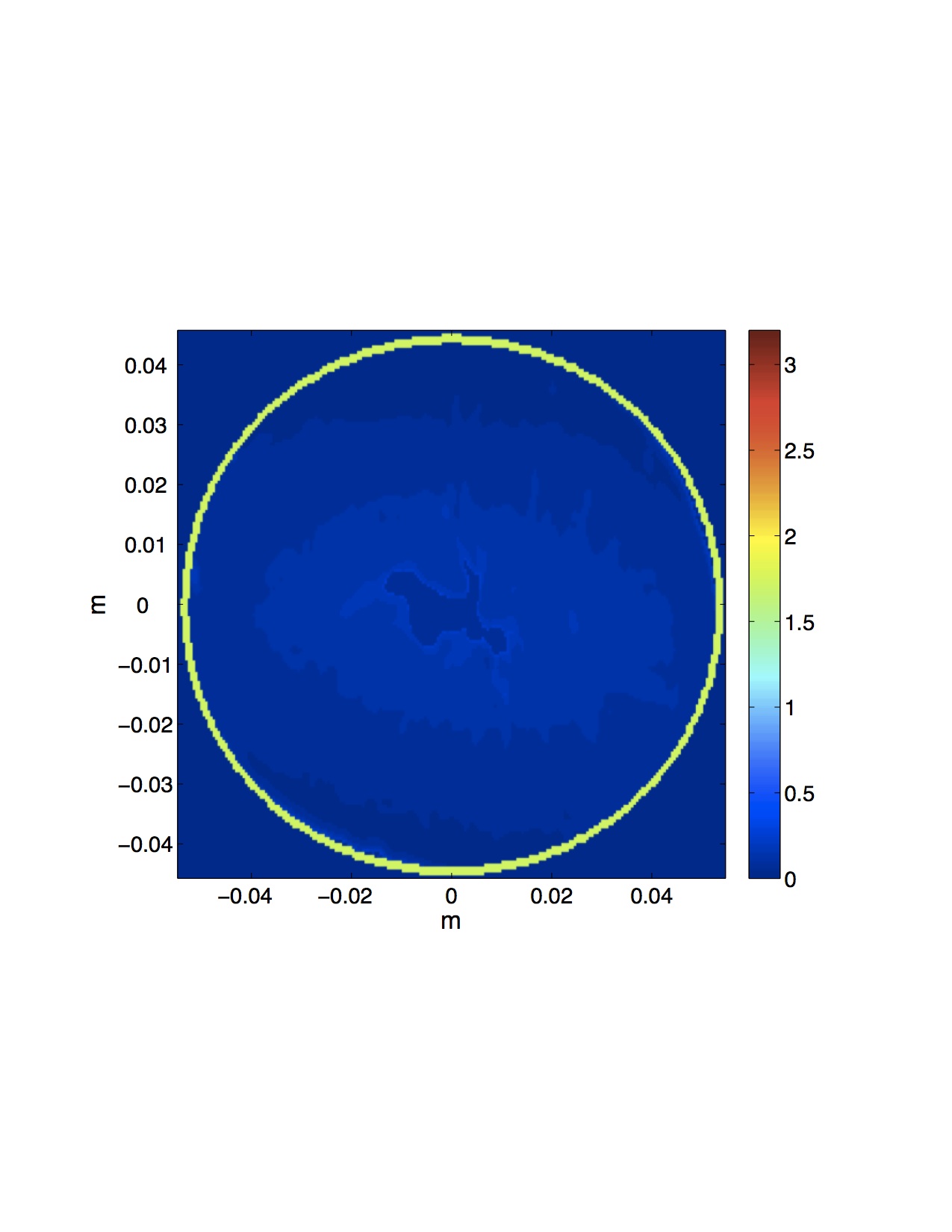, width=6.0cm}\\
$\ \ \ \ \ $ (a) & $\ \ \ \ \ $ (b) \\[0.35cm]
\psfig{figure=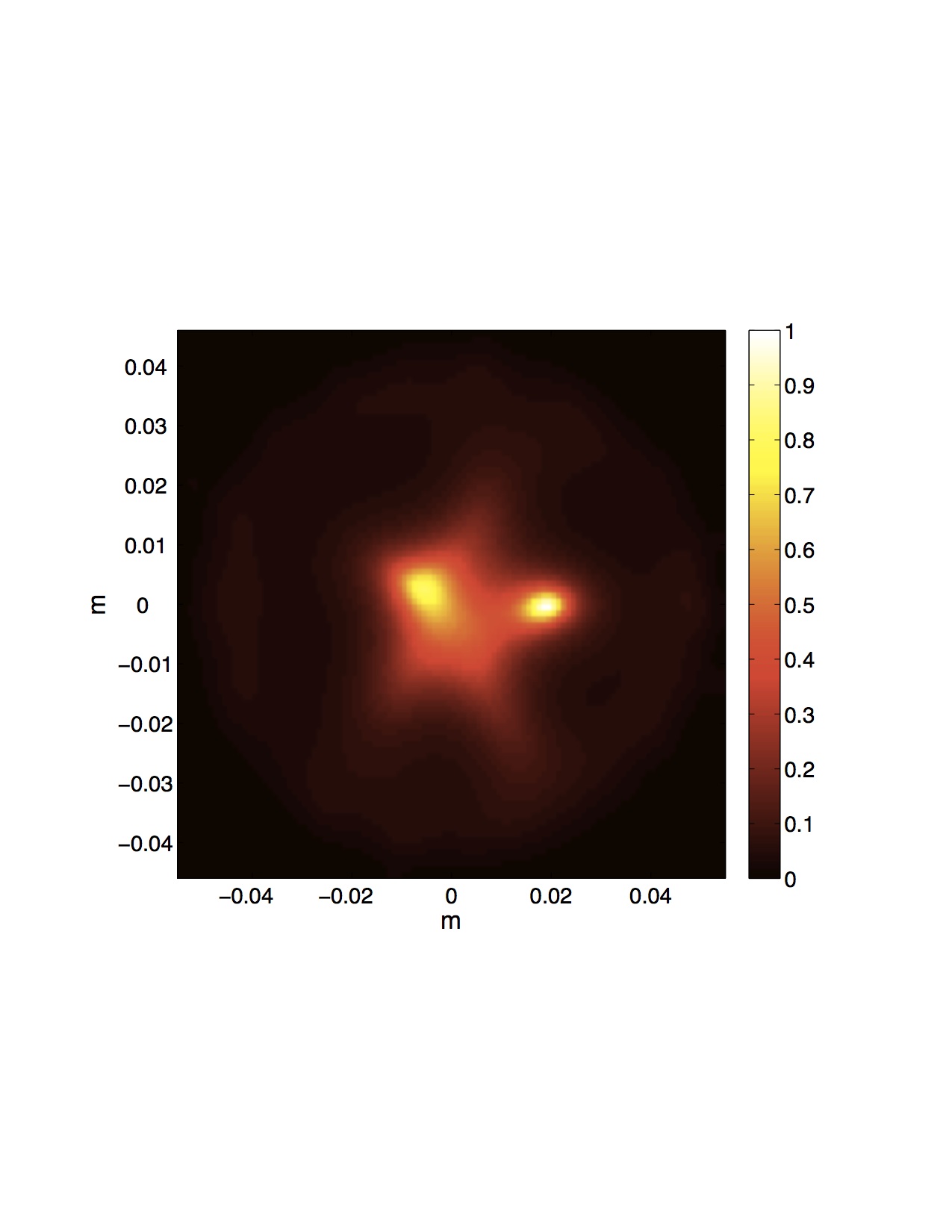, width=6.0cm} &
\psfig{figure=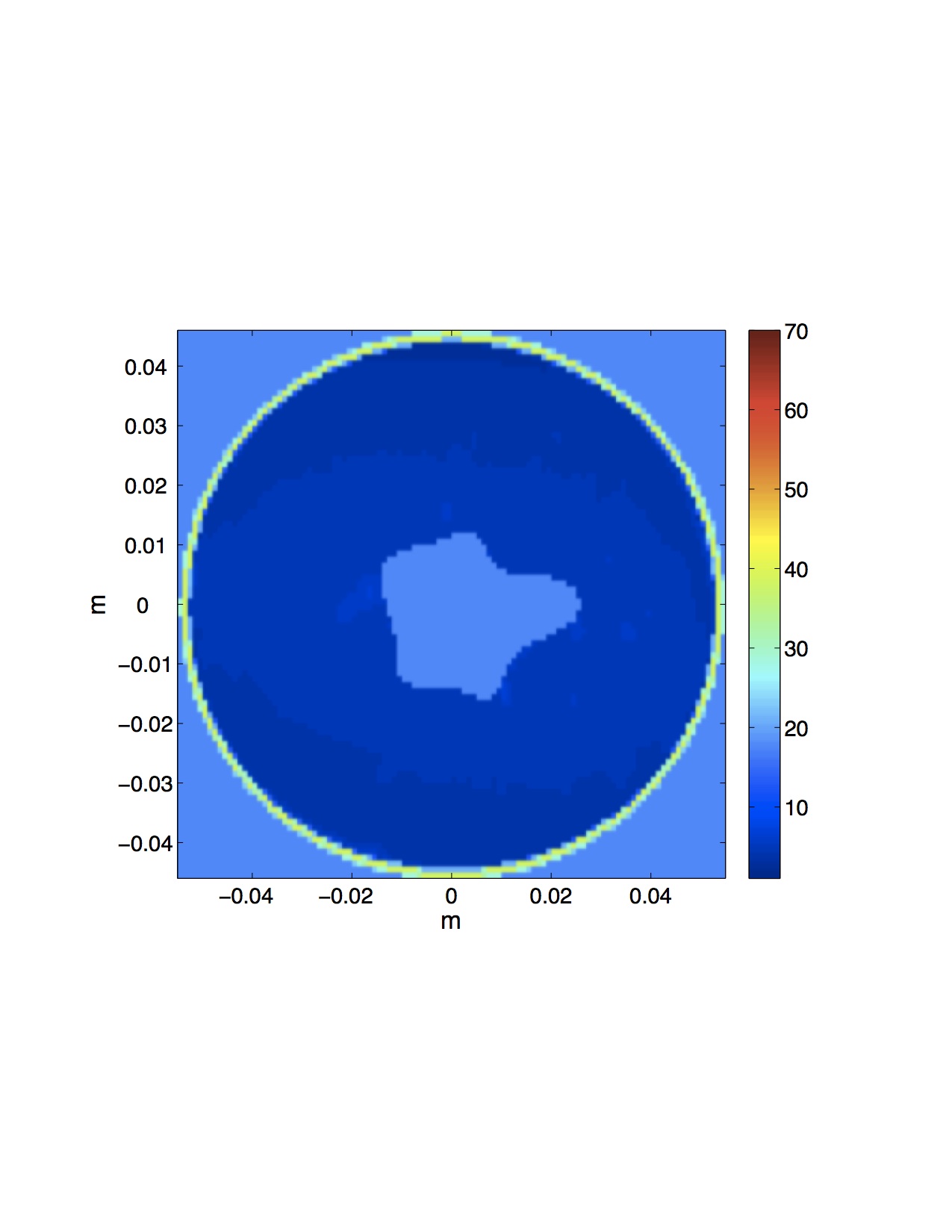, width=6.0cm}\\
$\ \ \ \ \ $ (c) & $\ \ \ \ \ $ (d) \\[0.35cm]
\psfig{figure=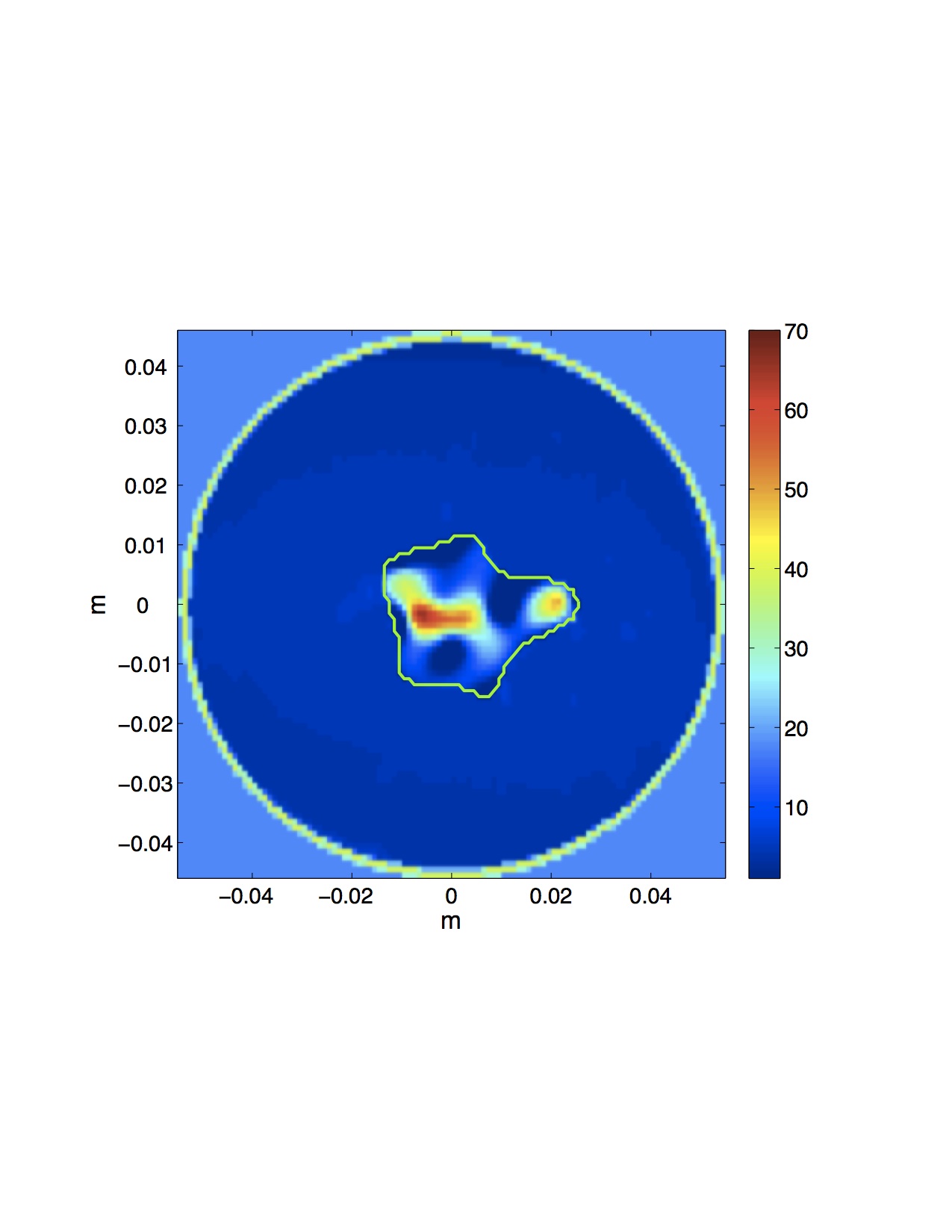, width=6.0cm} &
\psfig{figure=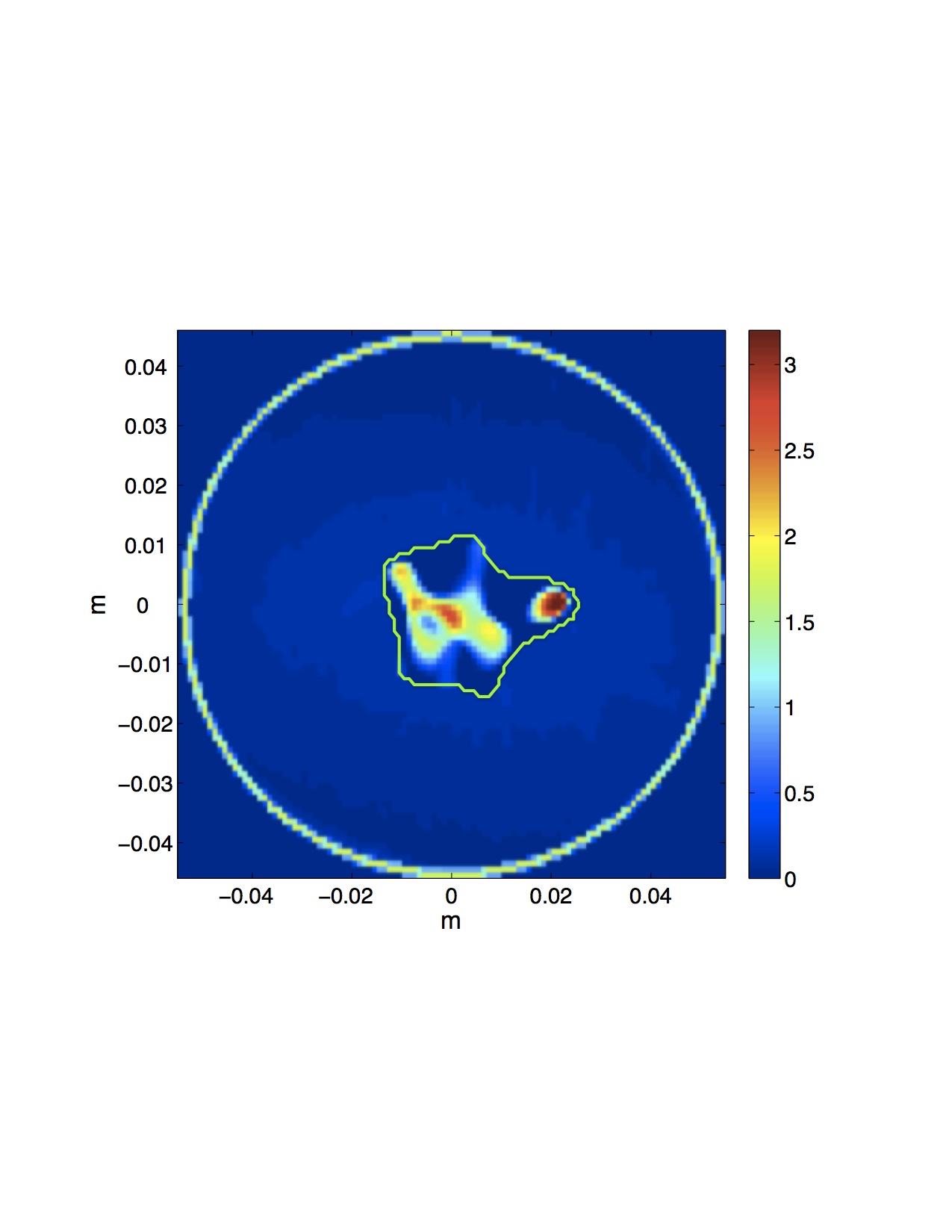, width=6.0cm}\\
$\ \ \ \ \ $ (e) & $\ \ \ \ \ $ (f) \\[0.35cm]
\end{tabular}
\caption{\label{breast_and_tumor}Reconstruction of the glandular and tumoral tissues by means of the hybrid approach. Panels (a) and (b):
pixel values of the relative permittivity $\varepsilon_r$ (a) and conductivity $\sigma$ (b) of the background utilized for the implementation of the NSLSM. Panel (c): visualization provided by the NSLSM. Panel (d): pixel values of the relative permittivity of the artificial background utilized for the implementation of the inhomogeneous CSI; the homogeneous inner domain $T$ having the same properties of the outer coupling
medium is obtained by applying an active contour technique to the visualization map shown in panel (c). Panel (e) and (f): relative permittivity (e) and
conductivity (f) as reconstructed by the inhomogeneous CSI; the region of reconstruction, delimited by a closed green line, coincides with the
investigation domain $T$ identified in panel (d).}
\end{center}
\end{figure}

In the second numerical example, we utilize the inhomogeneous CSI to investigate the nature of a small inhomogeneity detected by the NSLSM. In this case, while maintaining the coupling medium as before, the whole healthy breast slice, just as provided by the segmentation of the MRI image, is assumed as background, and the reconstruction is performed only on the pixels where the presence of the (prospective) tumor is highlighted (for a detailed analysis of the performance of the NSLSM in this framework, see \cite{bobrpa10}).
As before, panels (a), (b) and (c) of Figure \ref{tumor} are concerned with the NSLSM: (a) and (b) are the plots of the relative permittivity and
conductivity of the background, while (c) shows the NSLSM output, whence the investigation domain $T$ (containing the inhomogeneity to be analyzed) is extracted by using active contours. In panel (d) of Figure \ref{tumor} we plot the relative permittivity of the
artificial background utilized for the inhomogeneous CSI: the coupling medium is now inserted into the area $T$ detected by the NSLSM.
Finally, panels (e) and (f) show the reconstructions of the relative permittivity and conductivity provided by the inhomogeneous CSI inside $T$: again, a comparison with the true phantoms of Figure \ref{breast_exact} highlights that the quality of these reconstructions is rather good.

\begin{figure}%[H]
\begin{center}
\begin{tabular}{cc}
\psfig{figure=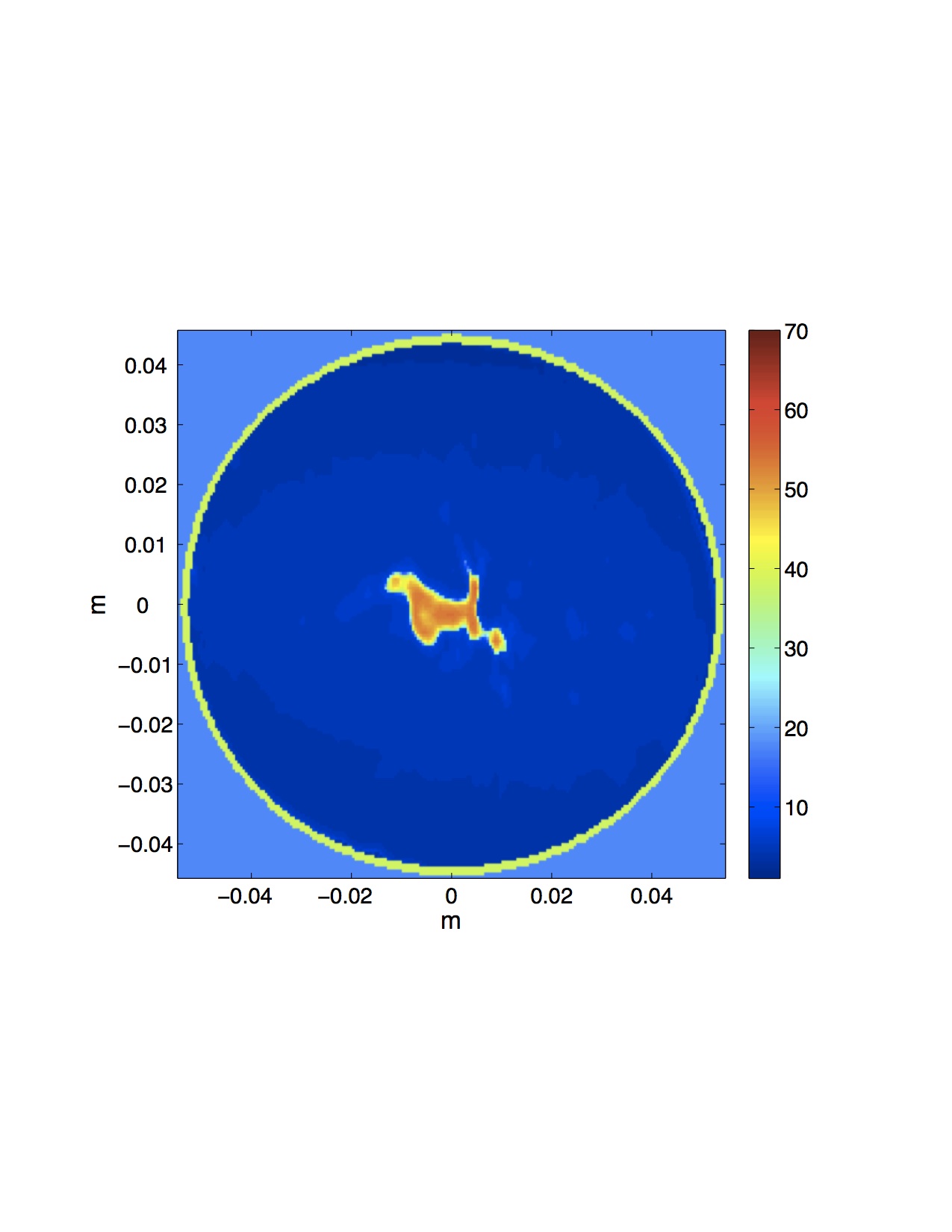, width=6.0cm} &
\psfig{figure=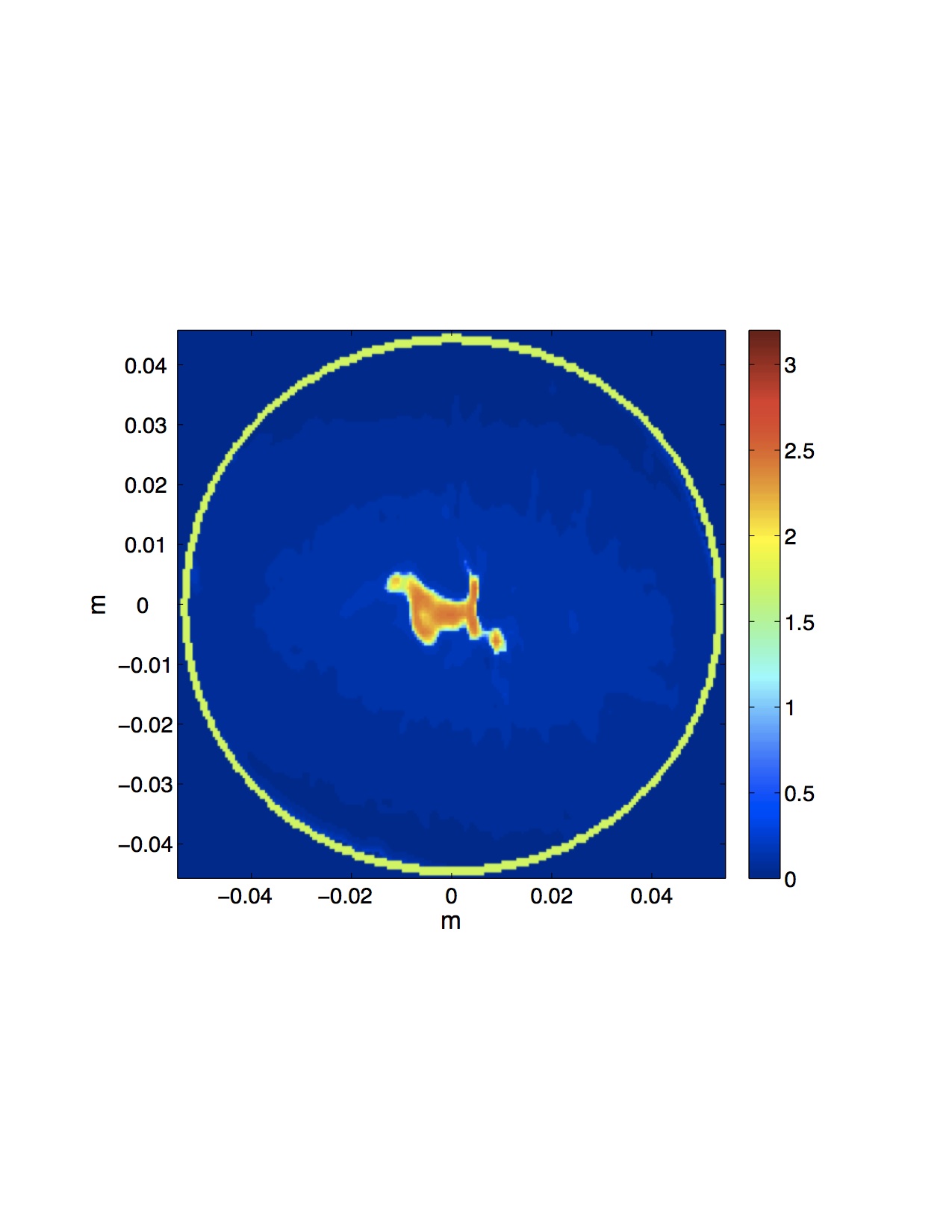, width=6.0cm}\\
$\ \ \ \ \ $ (a) & $\ \ \ \ \ $ (b) \\[0.35cm]
\psfig{figure=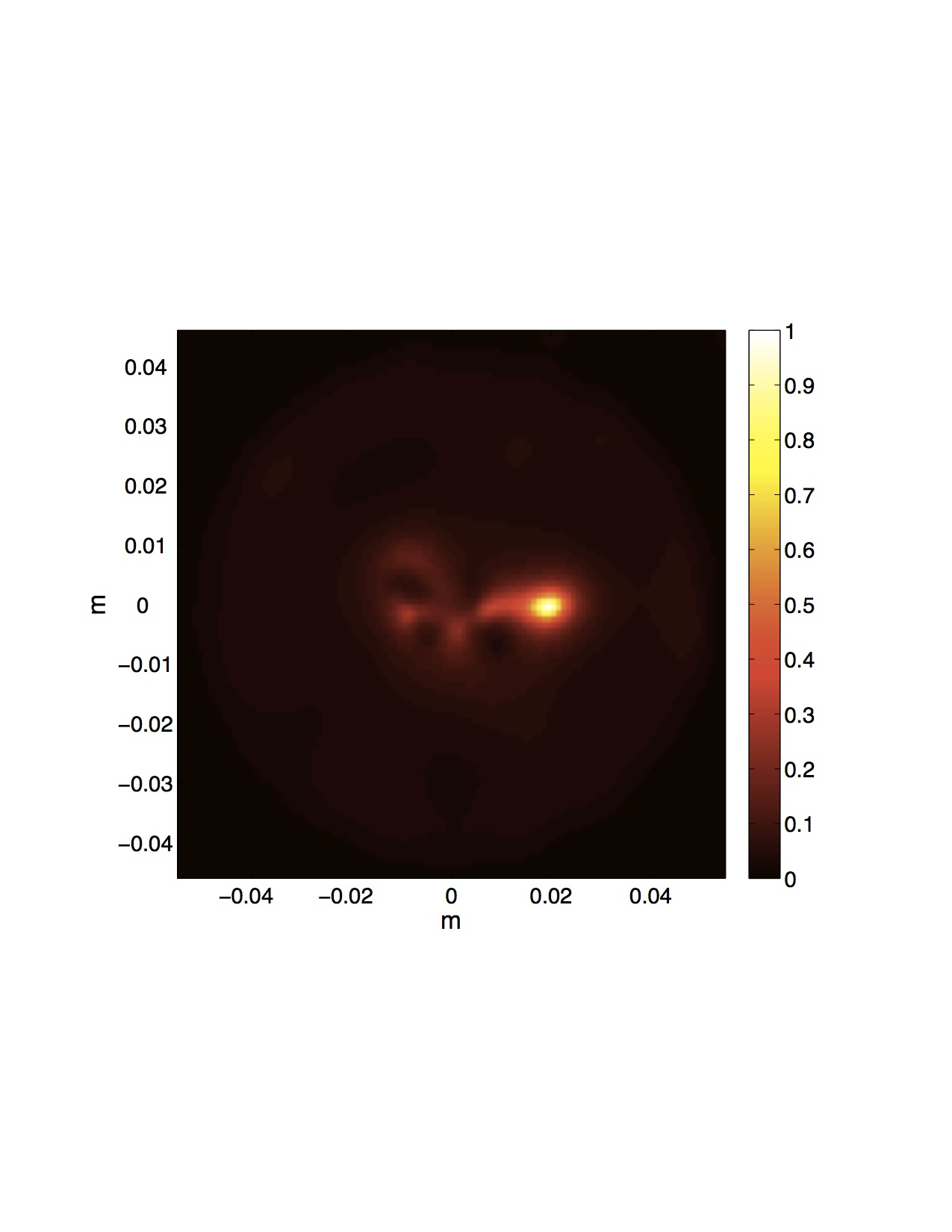, width=6.0cm} &
\psfig{figure=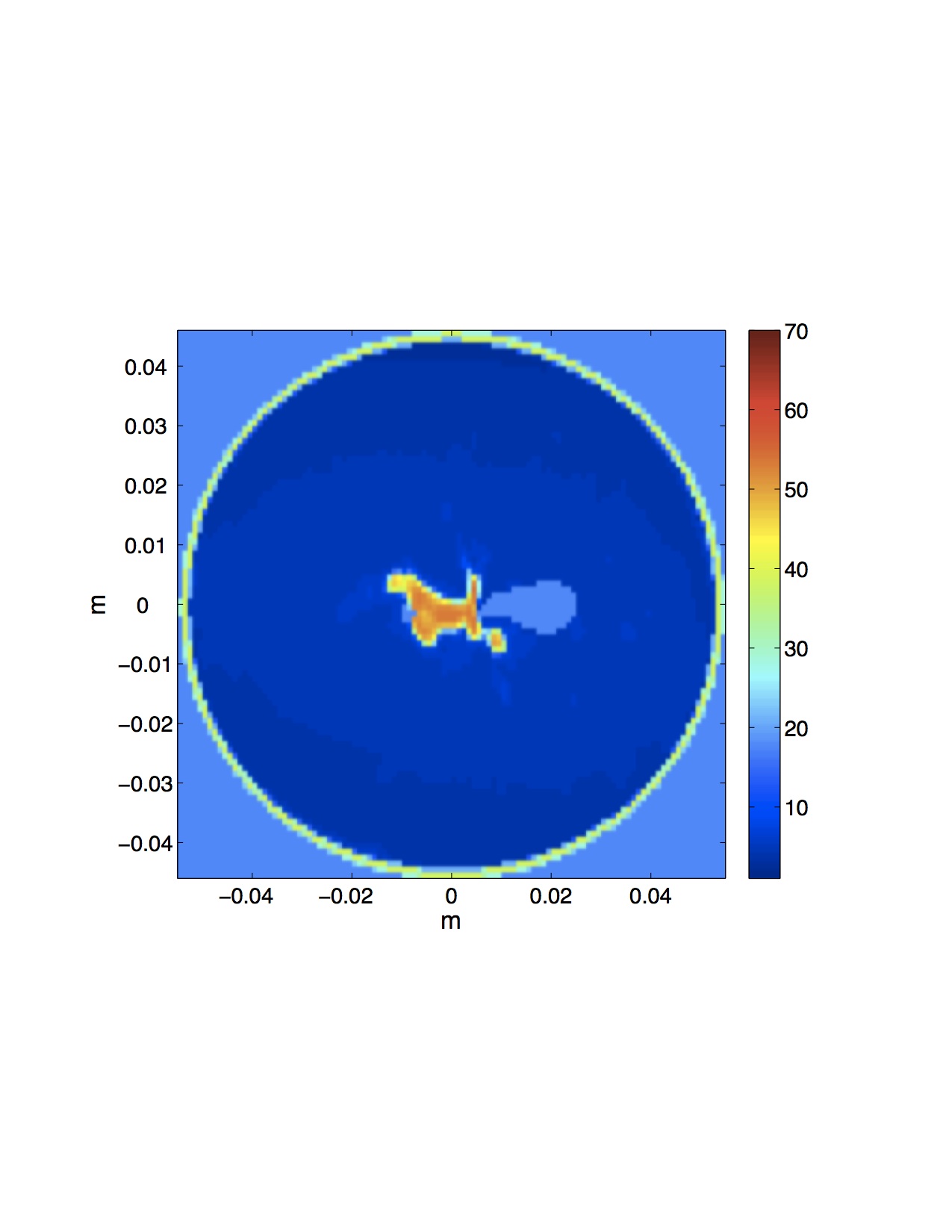, width=6.0cm}\\
$\ \ \ \ \ $ (c) & $\ \ \ \ \ $ (d) \\[0.35cm]
\psfig{figure=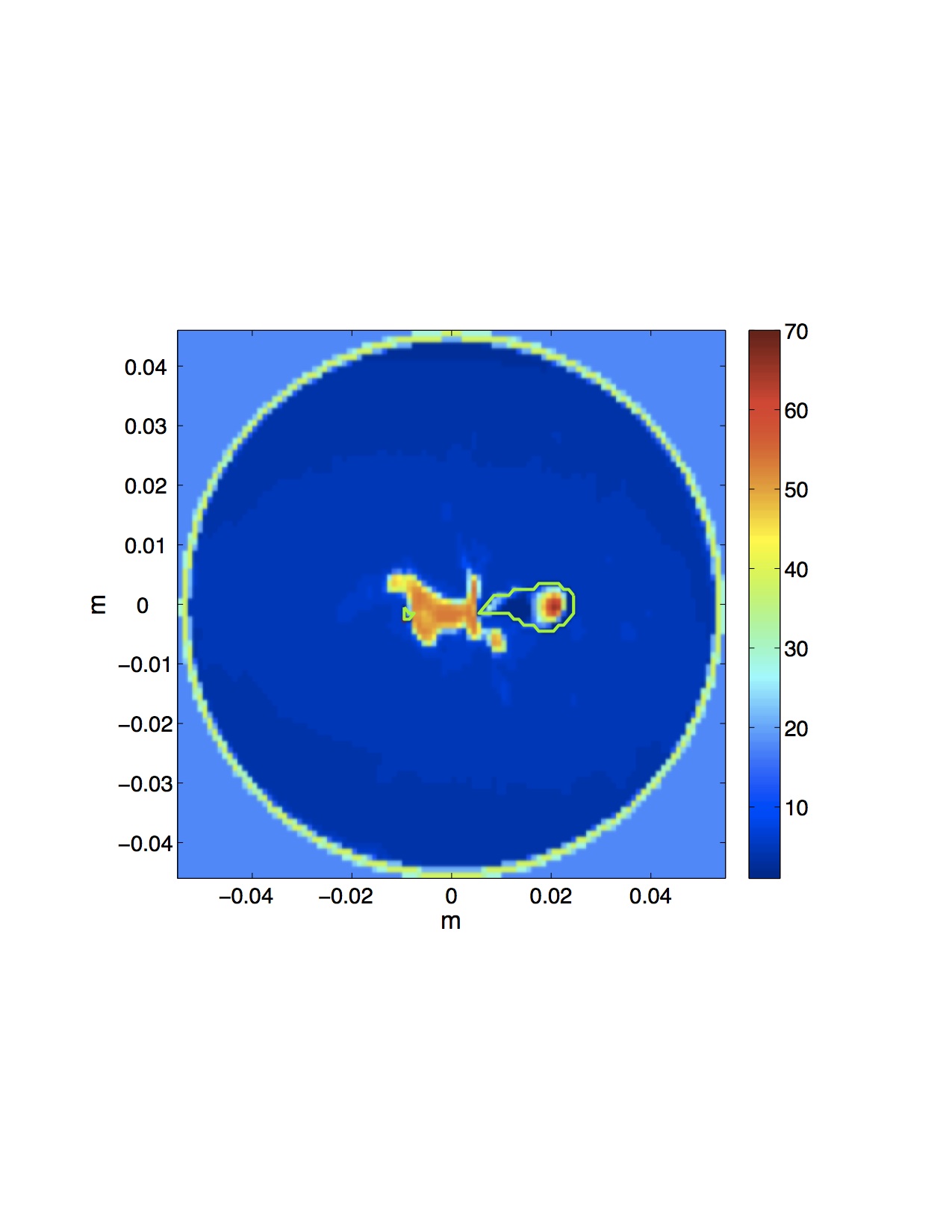, width=6.0cm} &
\psfig{figure=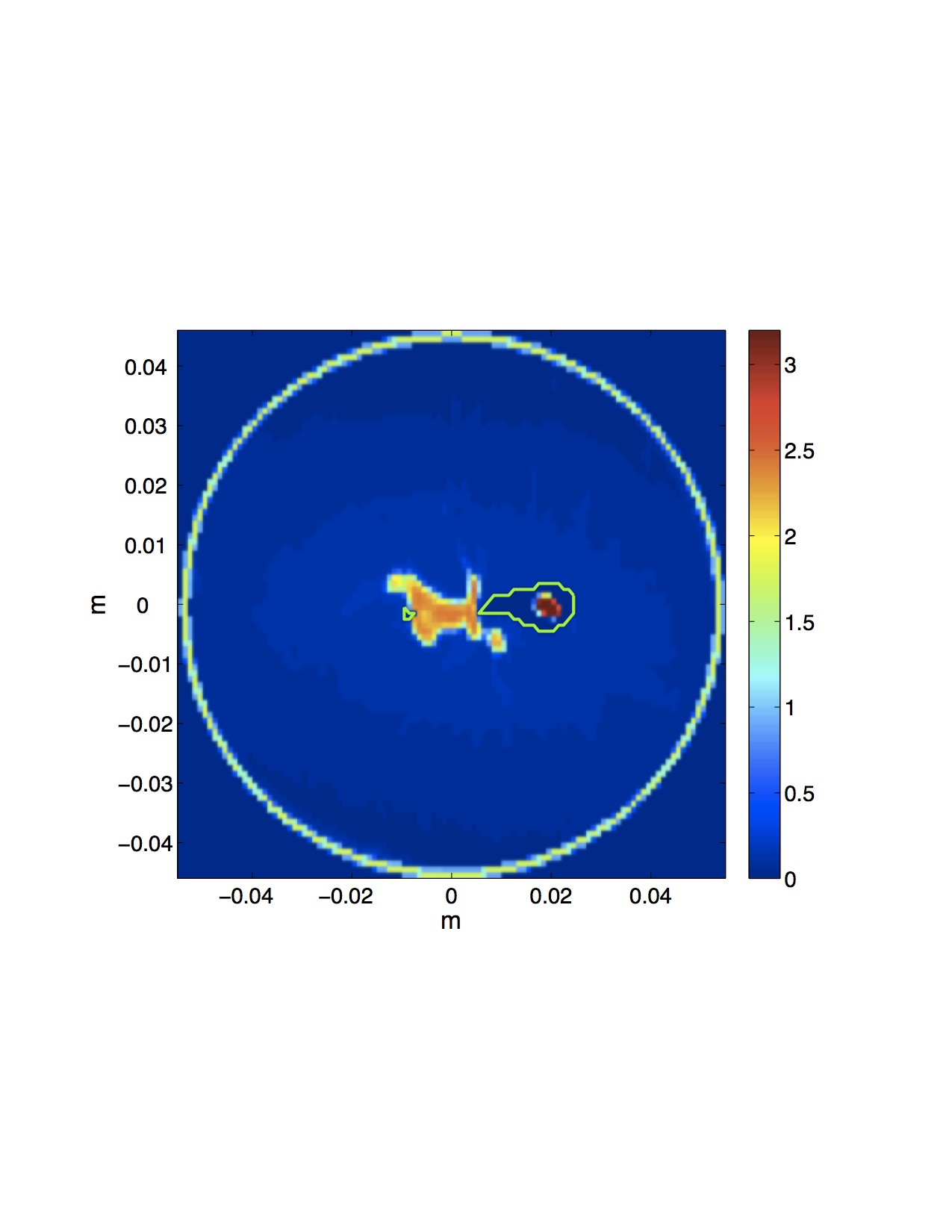, width=6.0cm}\\
$\ \ \ \ \ $ (e) & $\ \ \ \ \ $ (f) \\[0.35cm]
\end{tabular}
\caption{\label{tumor}Reconstruction of just the tumoral tissue by means of the hybrid approach. Panels (a) and (b):
pixel values of the relative permittivity $\varepsilon_r$ (a) and conductivity $\sigma$ (b) of the background utilized for the implementation of the
NSLSM. Panel (c): visualization provided by the NSLSM. Panel (d): pixel values of the relative permittivity of the artificial
background utilized for the implementation of the inhomogeneous CSI; the (non-connected) homogeneous inner domain $T$ having the same properties of the
outer coupling medium is obtained by applying an active contour technique to the visualization map shown in panel (c). Panel (e) and (f): relative
permittivity (e) and conductivity (f) as reconstructed by the inhomogeneous CSI;
the region of reconstruction, delimited by two closed green lines (a larger one on the right and a very small one on the left of the glandular tissue), coincides with the investigation domain $T$ identified in panel (d).}
\end{center}
\end{figure}

\section{Conclusions and further developments}\label{zioconcluso}

In this paper we have proposed a hybrid method to numerically solve a large class of two-dimensional inverse scattering problems.
According to this hybrid approach, a qualitative method, i.e., the NSLSM, is first applied to approximately identify the region of primary interest
inside an inhomogeneous background; then, a quantitative algorithm, i.e., the inhomogeneous CSI, is implemented only in this region to compute the
point values of the electric parameters (permittivity and conductivity) of the unknown scatterer.
Both steps are made possible by the knowledge of the inhomogeneous background: inserting such information in the whole procedure improves the quality
of the reconstruction and optimizes the number of data necessary for a satisfactory result, as shown by the numerical simulations
in Section \ref{zionumerico}.
The first (qualitative) step consists of a no-sampling implementation of the LSM for an inhomogeneous background, which can be performed according
to the guidelines recalled in Subsection \ref{basta} and detailed in the papers cited there. Instead, the second (quantitative) step requires a
preliminary theoretical investigation, in order to extend the usual Lippmann-Schwinger equation to the case of an inhomogeneous background.
Such a generalization has been obtained in Section \ref{zioinomogeneo} as a consequence
of the differential formulation of the scattering problem: under appropriate assumptions on the refractive index and on the magnetic permeability,
the new integral equation maintains its simplest form, i.e., involves no boundary terms across the discontinuities of the electric
parameters. Accordingly, the homogeneous and inhomogeneous Lippmann-Schwinger equations have the same structure: this has suggested the idea of
generalizing the CSI algorithm, originally based on the former equation, to the case of an inhomogeneous background, handled by the latter.

As regards possible further developments, a first task is trying to prove the exact equivalence between differential and integral 
formulation of the scattering problem (possibly by adopting a variational approach, i.e., by interpreting system 
(\ref{scatteringnear}) in a weak sense and looking for its solution in an appropriate Sobolev space, cf. \cite{ki11}): 
as pointed out in Section \ref{zioinomogeneo},
%introduced in Sections \ref{ziodifferenziale} and \ref{zioinomogeneo}: as pointed out there, 
such equivalence is likely to hold. As a generalization, the case of non-constant magnetic permeability $\mu(x)$ would also deserve 
an analogous investigation: the corresponding integral formulation would be more
complex, owing to the presence of boundary terms across the discontinuities of $\mu(x)$. Of course, the same problems can also be 
formulated in a three-dimensional framework. An ideal goal might be to parallel and generalize to the case of
an inhomogeneous background the results obtained in \cite{ma03} for a homogeneous one. This would allow extending our hybrid approach 
to very general three-dimensional inverse scattering problems, where the issues of a good balance between the amount of data and the 
number of unknowns, as well as the related problem of computational costs, are even more important than in two-dimensional set-ups.

\makeatletter
%\addcontentsline{toc}{chapter}{Bibliography}
%\section*{References}
\bibliography{Aramini_bibliography_2012}
\bibliographystyle{siam}
\end{document}